\documentclass[preprint, 11pt]{elsarticle}

\usepackage[T1]{fontenc}
\usepackage[utf8]{inputenc}
\usepackage[hmargin=0.9in]{geometry}
\usepackage[babel,german=quotes]{csquotes}
\usepackage{subfig}
\usepackage{graphicx}
\usepackage[colorlinks=true,citecolor=blue,urlcolor=blue]{hyperref}
\usepackage{caption}
\usepackage{amsmath}          
\usepackage{amsthm}           
\usepackage{amssymb}          
\usepackage{bm, amsfonts}     
\usepackage{xcolor}
\usepackage{fixmath}
\usepackage{tikz}
\usepackage{bm}
\usepackage{makecell}
\usepackage{mathdots}
\usepackage{algpseudocode}
\usepackage{algorithm}
\newtheorem{thm}{Theorem}
\newdefinition{cor}{Corollary}

\newdefinition{example}{Example}
\newproof{pf}{Proof}

\newcommand {\umin} {u^{\textup{min}}}
\newcommand {\umax} {u^{\textup{max}}}
\newcommand {\G} {{\mathcal G}}

\newcommand {\ds} {\,{\rm d}{\mathrm s}}

  \newcommand {\bF} {{\mathbf f}}

   \newcommand{\bff}{\mathbf{f}}

  \newcommand{\R}{\mathbb{R}}
  \newdefinition{rmk}{Remark}

  \newdefinition{defn}[rmk]{Definition}
  
  \newcommand{\pd}[2]{\frac{\partial #1}{\partial #2}}

  \newcommand{\td}[2]{\frac{\mathrm d #1}{\mathrm d #2}}

\newcommand{\beq}{\begin{equation}}
\newcommand{\eeq}{\end{equation}}

\newcommand{\bfn}{{\bf n}}

\newcommand{\bfx}{{\bf x}}

\newcommand {\jinN} {j\in\mathcal N_i}
\newcommand{\RK}{{\rm RK}}
\newcommand{\FE}{{\rm FE}}
\newcommand{\WENO}{{\rm WENO}}
\newcommand{\GMC}{{\rm GMC}}

\makeatletter
\def\ps@pprintTitle{%
  \let\@oddhead\@empty
  \let\@evenhead\@empty
  \def\@oddfoot{
    \footnotesize\itshape
    \hfill\today
  }%
  \let\@evenfoot\@oddfoot}
\makeatother

\begin{document}

\begin{frontmatter} 
  \title{Bound-preserving flux limiting for high-order 
    explicit Runge--Kutta time discretizations of hyperbolic
  conservation laws}

\author[TUDo]{Dmitri Kuzmin}
\ead{kuzmin@math.uni-dortmund.de}

\address[TUDo]{Institute of Applied Mathematics (LS III), TU Dortmund University\\ Vogelpothsweg 87,
  D-44227 Dortmund, Germany}

\author[KAUST]{Manuel Quezada de Luna\corref{cor1}}
\ead{manuel.quezada@kaust.edu.sa}
\cortext[cor1]{Corresponding author}

\author[KAUST]{David I. Ketcheson}
\ead{david.ketcheson@kaust.edu.sa}

\address[KAUST]{King Abdullah University of Science and Technology (KAUST)\\ Thuwal 23955-6900, Saudi Arabia}

\author[TUDo]{Johanna Gr\"ull}
\ead{jgruell@mathematik.tu-dortmund.de}

\journal{Elsevier}

\begin{abstract}
  We introduce a general framework for enforcing local or global
  maximum principles in high-order space-time discretizations of
  a scalar hyperbolic conservation law. We begin with sufficient
  conditions for a space discretization to be bound preserving
  (BP) and satisfy a semi-discrete maximum principle. Next, we
  propose a global monolithic convex (GMC) flux limiter which
  has the structure of a flux-corrected transport (FCT) algorithm
  but is applicable to spatial semi-discretizations and ensures
  the BP property of the fully discrete scheme for strong stability
  preserving (SSP) Runge-Kutta time discretizations. To circumvent
  the order barrier for SSP time integrators, we constrain the
  intermediate stages and/or the final stage of a general
  high-order RK  method using GMC-type
  limiters. In this work, our theoretical and numerical studies
  are restricted to explicit schemes which are provably
  BP for sufficiently small time steps. The new GMC limiting framework
  offers the possibility of relaxing the bounds of inequality
  constraints  to achieve higher accuracy
   at the cost of more stringent time step restrictions.
   The ability of the presented limiters to preserve global
   bounds and recognize well-resolved smooth solutions is 
  verified numerically for three representative RK methods
  combined with weighted essentially nonoscillatory (WENO)
  finite volume space discretizations of linear and nonlinear
  test problems in 1D.
\end{abstract}
\end{frontmatter}

\section{Introduction}

High-resolution numerical schemes for hyperbolic conservation laws
are commonly equipped with mechanisms that guarantee preservation
of local and/or global bounds. Bound-preserving (BP) second-order
approximations can be constructed, e.g.,
using flux-corrected transport (FCT) algorithms \cite{fct1,zalesak79}
or total variation diminishing (TVD) limiters \cite{harten1,harten2}.
If the spatial semi-discretization is BP and time
integration is performed using a strong stability preserving (SSP)
Runge-Kutta method \cite{david_book,ssprev}, the fully discrete
method can be shown to satisfy the corresponding maximum principle.

Discretization methods that use limiters to enforce preservation
of \emph{local} bounds can be at most second-order accurate
\cite{zhang2011maximum}. In contrast, the imposition of \emph{global} bounds does
not generally degrade the rates of convergence to smooth
solutions. For example, the use of weighted
essentially nonoscillatory (WENO) reconstructions 
in the context of finite volume methods 
or discontinuous Galerkin (DG) methods 
both with additional limiting makes 
it possible to construct positivity-preserving schemes of very
high order in space \cite{zhang2011maximum}. However, the
requirement that the time integrator be SSP imposes a
fourth-order barrier on the overall accuracy of explicit Runge-Kutta
schemes and a sixth-order barrier on the accuracy of
implicit ones. Moreover, only the first-order accurate
backward Euler method is
SSP for arbitrarily large time steps \cite{david_book}.

The general framework of spatially partitioned Runge-Kutta (SPRK)
methods \cite{Ketcheson2015,Ketcheson2013} makes it possible to
combine different time discretizations in an adaptive manner.
Following the design of limiters for space
discretizations, the weights of a flux-based SPRK method
\cite{Ketcheson2013} or blending functions of a partition of
unity finite element method (PUFEM) \cite{hpfem} can be chosen to
enforce global or local bounds.
The weights of the SPRK scheme proposed in \cite{arbogast} are
defined using a
WENO smoothness indicator which reduces the magnitude
of undershoots/overshoots but does not ensure positivity
preservation. Examples of BP
limiters for high-order time discretizations can be
found in \cite{arbogast,Timelim,Feng2019,lee2010,CG-BFCT}.
Perhaps the simplest approaches to limiting in time are predictor-corrector
algorithms based on the FCT methodology. They have already
proven their worth in the context of multistep methods \cite{lee2010},
Runge-Kutta time discretizations \cite{CG-BFCT}, and space-time
finite element schemes \cite{Feng2019}.

The limiting tools proposed in the present paper constrain
high-order spatial semi-discretizations and/or stages of a general
Runge-Kutta (RK) method to satisfy discrete maximum principles for
cell averages. As an alternative to FCT algorithms, which are
defined at the fully discrete level and inhibit convergence to
steady-state solutions, we design a limiter that exploits the
BP property of convex combinations and is based on less restrictive
constraints than the convex limiting method introduced in \cite{convex}.
We prove that the new limiter ensures the validity of a semi-discrete
maximum principle for the space discretization.
Non-SSP stages of a high-order RK method can be
constrained using the same algorithm. In this work, we 
incorporate  GMC limiters into high-order explicit WENO-RK
discretizations of 1D hyperbolic problems. The same flux correction
procedures can be used for other space discretizations, such as
high-order Bernstein finite element approximations
\cite{Bern-DGFEM1, Bern-DGFEM2,CG-BFCT}.

The rest of this paper is organized as follows. In Section
\ref{sec:gmcl}, we constrain a high-order finite volume
discretization of a scalar conservation law using
a new general criterion for convex limiting in space. In Section
\ref{sec:bp-theory}, we prove the BP property of semi-discrete
schemes satisfying this criterion. To our knowledge, this is
the first theoretical result of this kind. In Section
\ref{subsec:gmc}, we derive the new convex limiter for
space discretizations. In Section \ref{sec:space_time_limiting},
we show
how this limiter can be used to constrain the final stage and
intermediate stages of a high-order RK method.
We also mention the possibility of slope limiting
for bound-violating high-order reconstructions. We discuss
the choice and
implementation of appropriate limiting strategies for
three kinds of RK methods in Section \ref{sec:BPschemes},
perform numerical experiments in Section \ref{sec:num}, and
draw preliminary conclusions in Section \ref{sec:conclusions}.
The details of the algorithms that we
use in our 1D numerical examples are provided
in \ref{appendix:RK_methods} and \ref{appendix:one-dim_GMC_ExRK}.

\section{Flux limiting for spatial semi-discretizations}\label{sec:gmcl}

We consider
high-order finite volume approximations to a 
hyperbolic conservation law of the form
\beq
\pd{u}{t}+\nabla\cdot\mathbf{f}(u)=0
\qquad \mbox{in}\quad \Omega \subset\R^d,\quad d\in\{1,2,3\}.\label{conslaw}
\eeq
For simplicity, we assume that the domain  $\Omega$ is a hyperrectangle and prescribe periodic boundary conditions on
$\partial\Omega$. The initial condition is given by
\beq
 u(\cdot,0)=u_0 \qquad\mbox{in}\ \Omega.
\eeq
We discretize $\Omega$ using a  mesh consisting
of $N_h$ computational cells $K_i,\ i=1,\ldots,N_h$. The unit outward
normal $\mathbf{n}_{ij}$ is constant on
each face $S_{ij}=\partial K_i\cap\partial K_j$ of the
boundary $\partial K_i=\bigcup_{j\in\mathcal N_i}S_{ij}$.
The volume of $K_i$ and 
area of $S_{ij}$ are denoted by $|K_i|$ and $|S_{ij}|$, respectively. The
set $\mathcal N_i$ contains the indices of von Neumann
neighbors of cell $i$, i.e., the indices of neighbor
cells $K_j,\ j\neq i$ such that $|S_{ij}|>0$.
A~spatially varying numerical flux across the edge or face $S_{ij}$ is
denoted by $H(\hat u_i(\bfx),\hat u_j(\bfx),\mathbf{n}_{ij})$, where $\hat u_i(\bfx)$
and $\hat u_j(\bfx)$ are traces of polynomial reconstructions
in $K_i$ and $K_j$, respectively, evaluated at $\bfx\in S_{ij}$.
Henceforth, for simplicity in the notation, we omit the dependence of
$\hat u$ on $\bfx$.

Using the divergence theorem and approximating the flux
$\mathbf{f}(u)\cdot\mathbf{n}_{ij}$ across $S_{ij}$
by a suitably chosen numerical flux $H(\hat u_i,\hat u_j,\mathbf{n}_{ij})$, we obtain a system of ordinary differential
equations 
\beq
|K_i|\td{u_i}{t}=-\sum_{j\in\mathcal N_i}\int_{S_{ij}}
H(\hat u_i,\hat u_j,\mathbf{n}_{ij}) \ds,\qquad
i\in\{1,\ldots,N_h\}\label{fv-update}
\eeq
for the cell averages $u_i$.
The general form of the Lax-Friedrichs (LF) flux across $S_{ij}$ is
\beq\label{LLF_flux}
H(\hat u_i,\hat u_j,\mathbf{n}_{ij}) =  \frac{\bF(\hat u_j)+\bF(\hat u_i)}{2}
\cdot\mathbf{n}_{ij} -\frac{\lambda_{ij}}{2}(\hat u_j-\hat u_i),
\eeq
where $\lambda_{ij}$ is a strictly positive upper bound for the wave speed of
the Riemann problem associated with face $S_{ij}$. In the local
Lax-Friedrichs (LLF) method, $\lambda_{ij}$ is a local upper
bound. In the classical LF method, the same global upper bound
$\lambda_{ij}=\lambda$ is taken for all faces.

\begin{figure}
    \center
    \includegraphics[width=0.65\textwidth]{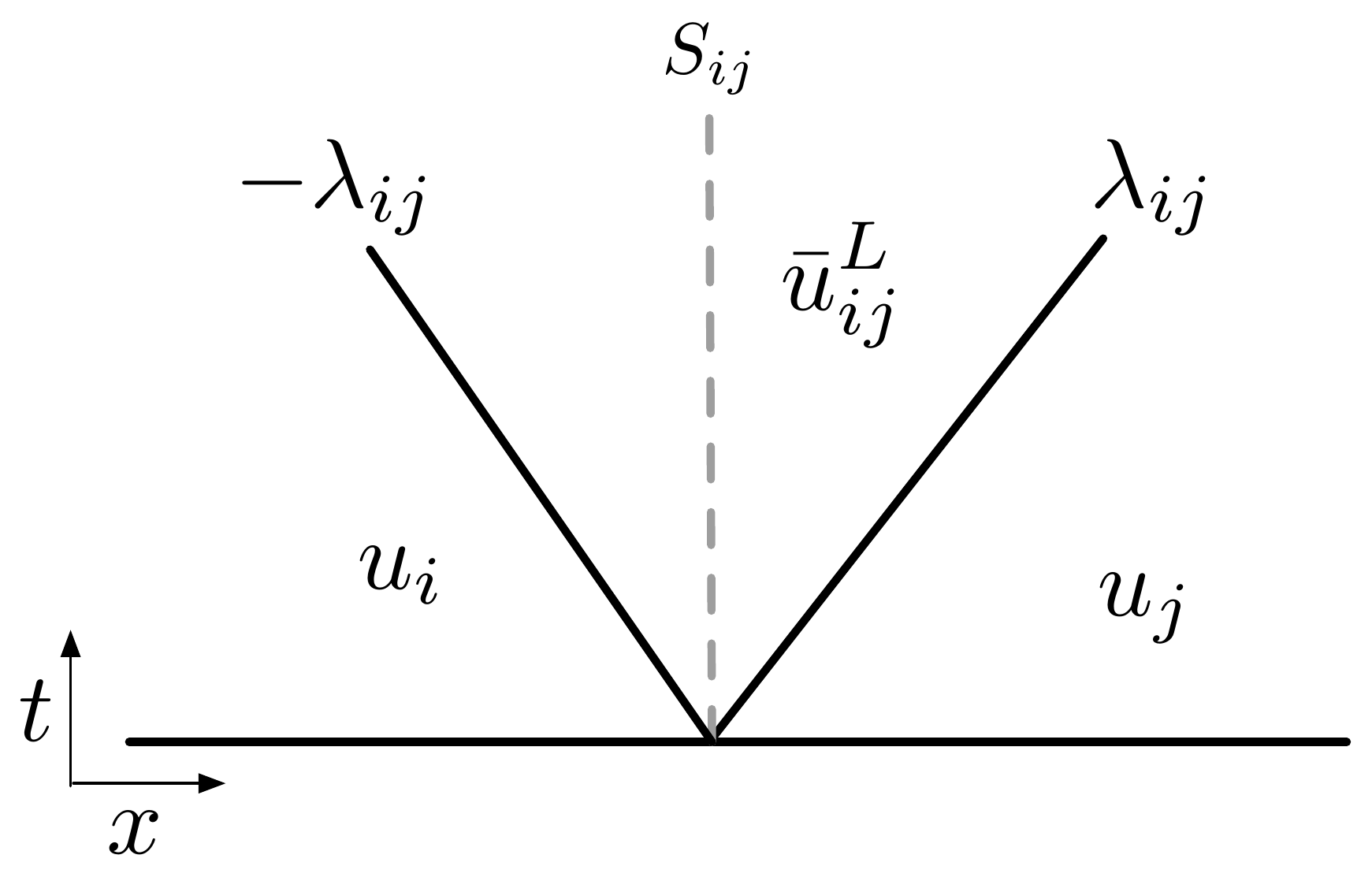}
    \caption{Structure of the Lax-Friedrichs approximate Riemann solution.
            The left and right states are the cell averages $u_i, u_j$,
            while the middle state is $\bar{u}^L_{ij}$.  The discontinuities
            separating these states travel at speeds $\pm \lambda_{ij}$.} \label{fig:riemann}
\end{figure}

\subsection{Bound-preserving schemes}\label{sec:bp-theory}

The first-order LLF scheme uses the cell averages $\hat u_i=u_i$ and $\hat u_j=u_j$ in \eqref{LLF_flux}.
The numerical flux $H(u_i,u_j,\mathbf{n}_{ij})$ can be derived from the conservation law
by assuming that the Riemann solution has the structure shown in Fig.~\ref{fig:riemann},
consisting of two traveling discontinuities.
The conservation law holds if the intermediate state (hereafter referred to
as the {\em bar state}) is given by \cite{leveque-fvm}
\beq\label{bar_states}
\bar u_{ij}^L=\frac{u_j+u_i}{2}
-\frac{\bF(u_j)-\bF(u_i)}{2\lambda_{ij}}\cdot
\mathbf{n}_{ij}.
\eeq
By the mean value theorem, the low-order (hence superscript $L$)
bar states $\bar u_{ij}^L$ satisfy \cite{convex}
\beq \label{lmpbar}
\min\{u_i,u_j\}\le \bar u_{ij}^L\le \max\{u_i,u_j\}.
\eeq
Using \eqref{bar_states} and the fact that
$\sum_{\jinN}|S_{ij}|\bff(u_i)\cdot\bfn_{ij} =0$, the
first-order LLF scheme can be written as
\beq\label{lo-barform}
|K_i|\td{u_i}{t} 
=\sum_{j\in\mathcal N_i}|S_{ij}|\lambda_{ij}(\bar u_{ij}^L-u_i)
=d_i(\bar u_i^L-u_i),
\eeq
where
\beq
\bar u_i^L=\frac{1}{d_i}\sum_{j\in\mathcal N_i}|S_{ij}|\lambda_{ij}\bar u_{ij}^L,
\qquad d_i=\sum_{j\in\mathcal N_i}|S_{ij}|\lambda_{ij}.
\eeq
The representation of the right-hand side in terms of the
jumps $\bar u_{ij}^L-u_i$
in the Riemann solution is known as the {\it fluctuation form} of the
finite volume scheme \cite{leveque-fvm}. The representation in terms
of the single jump $\bar u_i^L-u_i$ reveals that an equilibrium state
must be a convex combination of $\bar u_{ij}^L$.

\begin{defn}\cite{Guermond2016} 
  We say $\G\subset{\mathbb R}$ is an {\it invariant set} for the
  initial value problem \eqref{conslaw} if it holds that
  $$u(x,0)\in \G \ \ \forall \mathbf{x}\in\bar\Omega \implies u(x,t)\in \G \ \ \forall \mathbf{x}\in\bar\Omega$$
  for all $t>0$.
\end{defn}

\begin{defn}
We say a semi-discretization of \eqref{conslaw} is 
{\it bound preserving} (BP) w.r.t. $\G=[u^{\min},u^{\max}]$ if
$$u_i(0) \in \G\ \forall i\ \implies\
u_i(t) \in \G\ \forall i\ \forall t>0.$$
In this case we refer to $u^{\min}$ and $u^{\max}$ as \emph{global bounds}.
If $\G$ is an invariant set of the
 initial value problem, the scheme is called {\it invariant
  domain preserving} (IDP) \cite{Guermond2016,convex} or
(in case $\G$ is the positive orthant)
{\it positivity preserving} \cite{hu2013,ZhangShu2010,zhang2011maximum}.
\end{defn}

\begin{thm}[Semi-discrete maximum principle]\label{semi-dmp}
  Let 
  $$G=\{w\in\mathbb{R}^{N_h}\,:\, w_i\in[\alpha,\beta],\ i=1,\ldots,N_h\}$$
  and consider the initial value problem
  \begin{align} \label{ivp}
    u_i'(t)  = a_i(u)(g(u)_i - u_i),\qquad  u_i(0) & = u_i^0,
  \end{align}
  where $g$ is such that
  \begin{align} \label{uinG}
    u \in G \implies g(u) \in G
  \end{align}
  and
  \begin{align} \label{lambda-bound}
    0 \le a_i(u) \le C \ \ \ \forall u \in G.
  \end{align}
  Furthermore, assume that \eqref{ivp} has a unique solution for
  all $t>0$ for all $u^0 \in G$.  Then the solution satisfies
  \beq
    u(t) \in G \ \ \forall \ \ t\ge 0. \label{mp-result}
  \eeq
\end{thm}

\begin{proof}
If $a_i$ is independent of $i$, then this is a direct application of
Nagumo's lemma (see e.g. \cite[Theorem~4.1]{deimling} or
\cite[Theorem~3.1]{blanchini}.  For the general
case, observe that for $u\in G$ and $\epsilon<1/C$ we have
$$
    u_i(t) + \epsilon u_i'(t) = (1-\epsilon a_i(u(t)))u_i + \epsilon a_i(u(t)) g(u)_i \in [\alpha,\beta].
$$
The result then follows from \cite[Lemma 2]{horvath2005}.
\end{proof}
The assumption of existence and uniqueness above can be avoided by invoking
other assumptions such as Lipschitz continuity; see \cite{deimling}.

Applying the theorem above to the first-order Lax-Friedrichs
scheme \eqref{lo-barform}, we obtain
\begin{cor}
   Let global bounds $\umin, \umax$ be given and let 
    $$G=\{w\in\mathbb{R}^{N_h}\,:\, w\in[\umin,\umax],\ i=1,\ldots,N_h\}.$$
    Assume that $\lambda_{ij}$ is uniformly bounded on $G$.
    Then the scheme \eqref{lo-barform} is bound preserving w.r.t.
    $[\umin,\umax]$.
\end{cor}
\begin{proof}
    The first-order scheme \eqref{lo-barform} is 
    of the form \eqref{ivp} with $a_i=d_i/|K_i|$ and $g(u)_i = \bar{u}_i$.
    The assumption on $\lambda_{ij}$ ensures that
    condition \eqref{lambda-bound} holds.
    The validity of \eqref{uinG} follows from the fact that
    $\bar u_i=\bar u_i^L$ is a convex combination of the bar states
    $\bar u_{ij}^L$ satisfying~\eqref{lmpbar}.
\end{proof}

Let $\mathcal G=[u^{\min},u^{\max}]$ be an invariant set of
problem \eqref{conslaw}. Define the range
$\mathcal G_i=[u_i^{\min},u_i^{\max}]\subseteq\mathcal G$
of numerically admissible states using (global or
local) bounds $u_i^{\min}$ and $u_i^{\max}$ such that
  \beq\label{local-bounds}
u^{\min}\le 
u_i^{\min}\le\min_{j\in\mathcal V_i}u_j\le u_i\le
\max_{j\in\mathcal V_i}u_j\le u_i^{\max}\le u^{\max},
\eeq
where $\mathcal V_i$ is the set containing the index $i$ and
the indices of all cells that share a vertex with~$K_i$.

\begin{defn}
  A BP scheme of the form \eqref{ivp} is called {\it local extremum
    diminishing} (LED) \cite{jameson1,jameson2}  w.r.t.
  $\mathcal G_i(t)=[u_i^{\min}(t),u_i^{\min}(t)]\subset
[u^{\min},u^{\max}]$ at $t\ge 0$ if
  $$u_j\in \G_i \ \ \forall j\in\mathcal V_i
  \implies g_i(u)\in \G_i.$$
\end{defn}
Because of \eqref{lmpbar}, the first-order scheme \eqref{lo-barform} is
LED w.r.t. any $\G_i=[u_i^{\min},u_i^{\min}]$ such that $u_i,u_j\in\G_i$.

\subsection{Bound-preserving limiters}\label{subsec:gmc}

The accuracy of the low-order approximation can be
improved by using
anfidiffusive fluxes $F_{ij}^*$ to correct the bar states
$\bar u_{ij}^L$.
The resulting generalization
of \eqref{lo-barform} and \eqref{fv-update} is given by
\beq\label{afc-barform}
|K_i|\td{u_i}{t}=
\sum_{j\in\mathcal N_i}|S_{ij}|\lambda_{ij}(\bar u_{ij}^*-u_i)=d_i(\bar u_i^*-u_i).
\eeq
The flux-corrected intermediate states $\bar u_{ij}^*$ and $\bar u_i^*$
are defined as follows:
\beq\label{barstar}
\bar u_{ij}^*=\bar u_{ij}^L+\frac{F_{ij}^*}{|S_{ij}|\lambda_{ij}},\qquad
\bar u_i^*=\frac{1}{d_i}\sum_{j\in\mathcal N_i}|S_{ij}|\lambda_{ij}\bar u_{ij}^*
=\bar u_i^L+\frac{1}{d_i}\sum_{j\in\mathcal N_i}F_{ij}^*.
\eeq
For example, the high-order finite volume scheme \eqref{fv-update}
is recovered for $F_{ij}^*=F_{ij}$, where
\beq\label{Fijdef}
F_{ij}=\int_{S_{ij}}[H(u_i,u_j,\mathbf{n}_{ij})
-H(\hat u_i,\hat u_j,\mathbf{n}_{ij})]\ds.
\eeq
Introducing a free parameter $\gamma>0$, we
limit the fluxes $F_{ij}$ in a manner
which  guarantees that
\beq\label{BPconstr2}
(1+\gamma)(u_i^{\min}-u_i)\le \bar u_i^*-u_i\le(1+\gamma)(u_i^{\max}-u_i).
\eeq
The resulting scheme is BP for any $\gamma>0$. Indeed,
\eqref{afc-barform} can be written in the form \eqref{ivp} with
\beq\label{afc-coeff}
 a_i=(1+\gamma)\frac{d_i}{|K_i|},\qquad
g_i(u)=u_i+\frac{\bar u_i^*-u_i}{1+\gamma}\in[u_i^{\min},u_i^{\max}].
\eeq

Let us now discuss the way to calculate the limited counterparts
$F_{ij}^*$ of the fluxes $F_{ij}$.
Recalling the definition \eqref{barstar} of $\bar u_i^*$, we find that
the inequality constraints \eqref{BPconstr2} hold if and only if
  \beq\label{fluxconstr}
 Q_i^-\le 
\sum_{j\in\mathcal N_i}F_{ij}^*\le Q_i^+,
\eeq
where
 \beq\label{bounds_glob}
  Q_i^-=d_i[(u_i^{\min}-\bar u_i^L)+\gamma (u_i^{\min}-u_i)],\quad
  Q_i^+=d_i[(u_i^{\max}-\bar u_i^L)+\gamma (u_i^{\max}-u_i)].
  \eeq
  The use of $\gamma>0$ in  \eqref{bounds_glob}
  relaxes the bounds $Q_i^{\pm}$ in a way that makes
  the flux constraints \eqref{fluxconstr} less restrictive. As a result,
  higher accuracy can be achieved using larger values of $\gamma$.
  However, the CFL condition for explicit schemes becomes more
  restrictive, as we show in Section \ref{sec:space_time_limiting}
  below.
  
  By definition of $d_i$ and $\bar u_i^L$, the bounds $Q_i^\pm$
of the flux constraints   \eqref{fluxconstr} 
  can be decomposed into
 \beq\label{bounds_loc}
  Q_{ij}^-=|S_{ij}|\lambda_{ij}[(u_i^{\min}-\bar u_{ij}^L)+\gamma(u_i^{\min}-u_i)],\quad
  Q_{ij}^+=|S_{ij}|\lambda_{ij}[(u_i^{\max}-\bar u_{ij}^L)+\gamma (u_i^{\max}-u_i)]. \eeq
  Localized flux limiting algorithms \cite{cotter,Guermond2018,convex,CG-BFCT} use such decompositions to replace inequality constraints for sums of fluxes by sufficient conditions that make it possible to constrain each flux independently. The {\it local monolithic convex} (LMC) limiting algorithm proposed in \cite{convex} yields
\beq\label{fij-bar}
F_{ij}^*=\begin{cases}
\min\,\{Q_{ij}^+,F_{ij},-Q_{ji}^-\}
& \mbox{if}\ F_{ij}>0,\\[0.25cm]
\max\,\{Q_{ij}^-,F_{ij},-Q_{ji}^+\}
& \mbox{otherwise},
\end{cases}
\eeq
where $F_{ij}$ is the antidiffusive flux defined by \eqref{Fijdef}.
The LMC limiter is designed to guarantee that
\beq
Q_{ij}^-\le F_{ij}^*\le Q_{ij}^+,\qquad F_{ji}^*=-F_{ij}^*,\qquad
Q_{ji}^-\le F_{ji}^*\le Q_{ji}^-.
\eeq
 Hence, the flux-corrected scheme is
conservative and satisfies conditions \eqref{fluxconstr}
for $i=1,\ldots,N_h$.

To avoid a potential loss of accuracy due to localization of the flux
constraints, we introduce a  {\it global monolithic convex} (GMC) 
limiter that enforces  \eqref{fluxconstr} using the following algorithm:
\begin{enumerate}
\item Calculate the sums of positive and negative antidiffusive fluxes
  \beq\label{p-def}
  P_i^+=\sum_{j\in\mathcal N_i}\max\{0,F_{ij}\},\qquad
  P_i^-=\sum_{j\in\mathcal N_i}\min\{0,F_{ij}\}.
  \eeq
\item Use the sums $P_i^\pm$ and the bounds $Q_i^\pm$ defined by
  \eqref{bounds_glob} to calculate
\beq  
R_i^+=\min\left\{1,\frac{Q_i^+}{P_i^+}\right\},\qquad
R_i^-=\min\left\{1,\frac{Q_i^-}{P_i^-}\right\}.
\eeq
\item Calculate the limited antidiffusive fluxes
  $F_{ij}^*=\alpha_{ij}F_{ij}$, where
\beq \label{alpha-def}
\alpha_{ij}=\begin{cases}
  \min\{R^+_i,R^-_j\} & \mbox{if}\ F_{ij}>0,\\
  1 & \mbox{if}\ F_{ij}=0,\\
  \min\{R^-_i,R^+_j\} & \mbox{if}\ F_{ij}<0.
\end{cases}
\eeq
\end{enumerate}
This limiting strategy is based on Zalesak's FCT algorithm \cite{zalesak79}
but the GMC bounds \eqref{bounds_glob} are independent of the time step, and
the validity of a discrete maximum principle is guaranteed for any strong
stability preserving (SSP) Runge-Kutta time integrator
\cite{david_book,ssprev}. In Section
\ref{sec:space_time_limiting}, we verify this
claim for a single forward Euler step
and discuss flux limiting for general Runge-Kutta stages.


\section{Flux limiting for Runge-Kutta methods}\label{sec:space_time_limiting}

Given a BP space discretization of the form \eqref{ivp} and a
time step $\Delta t>0$, the first-order accurate forward Euler (FE)
method advances the cell averages $u_i,\ i=1,\ldots,N_h$ in time as follows:
\beq\label{FEstep}
u^{n+1}=u_i^n+\nu_i(\bar u_i^n-u_i^n)
=(1-\nu_i)u_i^n+\nu_i\bar u_i^n,
\eeq
where $\nu_i=\Delta t a_i>0$ is a generalized `CFL' number.
If $\nu_i\le 1$ then $u_i^{n+1}$ is a convex combination of $u_i^n$
and $\bar u_i^n$. It follows that $u_i^{n+1}\in [u_i^{\min},u_i^{\max}]$
if $u_i^n,\bar u_i^n\in [u_i^{\min},u_i^{\max}]$. In view of \eqref{afc-coeff},
the fully discrete version of the semi-discrete scheme
\eqref{afc-barform} equipped with the LMC or GMC flux limiter
is BP if 
\beq\label{cfl}
(1+\gamma)
\frac{\Delta td_i}{|K_i|}\le 1.
\eeq
The low-order LLF scheme \eqref{lo-barform} and Zalesak's FCT method
are BP for time steps satisfying \eqref{cfl} with $\gamma=0$. The
FCT fluxes $F_{ij}^*$ are calculated using algorithm
\eqref{p-def}--\eqref{alpha-def} with global bounds
\beq\label{bounds_fct}
Q_i^-=\frac{|K_i|}{\Delta t}(u_i^{\min}-u_i^{\rm FE}),\qquad
Q_i^+=\frac{|K_i|}{\Delta t}(u_i^{\max}-u_i^{\rm FE})
\eeq 
depending on $u_i^{\rm FE}=u_i^n+\frac{\Delta t}{|K_i|}
d_i(\bar u_i^L-u_i^n)$. A localized version of FCT is defined by
\eqref{fij-bar} with \cite{LimiterDG}
 \beq\label{bounds_locFCT}
  Q_{ij}^-=\frac{|K_i|\,|S_{ij}|}{\Delta t|\partial K_i|}(u_i^{\min}-u_i^{\rm FE}),\qquad
  Q_{ij}^+=\frac{|K_i|\,|S_{ij}|}{\Delta t|\partial K_i|}(u_i^{\max}-u_i^{\rm FE}).
  \eeq
 Flux limiting techniques of this kind were considered, e.g., in
 \cite{Guermond2018,convex,LimiterDG}. The main advantage of
 localized convex limiting lies in its applicability to
 systems. Since the main focus of the present paper is on
 scalar conservation laws, we restrict further discussion
 to global MC and FCT limiters.

 \begin{rmk}
   Since the FCT bounds \eqref{bounds_fct} and \eqref{bounds_locFCT}
   depend on the time step and on the low-order predictor
   $u_i^{\rm FE}$, the resulting nonlinear schemes have no
   semi-discrete or steady-state counterparts.
   \end{rmk}
 
 \subsection{High-order Runge-Kutta methods}
  Let us rewrite \eqref{afc-barform} as
 \beq\label{afcRK}
 |K_i|\td{u_i}{t}=-\sum_{j\in\mathcal N_i}|S_{ij}|H_{ij}^*(\hat u),\qquad
H_{ij}^*\left(\hat u\right)=H(u_i,u_j,\mathbf{n}_{ij})-
\frac{F_{ij}^*(\hat u)}{|S_{ij}|},
\eeq
where $\hat u$ is a high-order piecewise-polynomial reconstruction from
the cell averages $u_1,\ldots, u_{N_h}$.

 We now discretize \eqref{afc-barform} in time using an explicit
 $M$-stage Runge-Kutta (RK) method with coefficients denoted by the vectors
 $c, b$ and matrix $A$.
 The RK stage approximations
$y_i^{(m)}\approx u_i(t^n+c_m\Delta t)$ to the cell averages
are given by 
\begin{subequations}\label{stage-butcher}
\begin{align}
y_i^{(1)} &= u_i^n,\label{rk-y1} \\
y_i^{(m)} &= u_i^n-\frac{\Delta t}{|K_i|}\sum_{j\in\mathcal N_i}
|S_{ij}|\sum_{s=1}^{m-1}a_{ms}H_{ij}^*\left(\hat{y}^{(s)}\right),
\qquad m=2,\ldots, M.\label{rk-ym}
\end{align}
\end{subequations}
The RK solution $u_i^{\rm RK}\approx u_i(t^n+\Delta t)$ is defined as
\begin{align}
  u_i^{\rm RK}=u_i^n-\frac{\Delta t}{|K_i|}\sum_{j\in\mathcal N_i}|S_{ij}|H_{ij}^{\rm RK},\qquad
 H_{ij}^{\rm RK} = \sum_{m=1}^Mb_mH_{ij}^*\left(\hat{y}^{(m)}\right).
  \label{rk-update}
\end{align}

Note that the intermediate cell averages $y_i^{(2)},\ldots,
y_i^{(M)}$ are generally not BP even for the low-order
LLF scheme, i.e., in the case $\hat y_i^{(m)}\equiv y_i^{(m)}$.
However, if a RK stage can be written as a convex combination
of forward Euler steps \eqref{FEstep}, then it inherits their BP properties.
It is therefore useful to distinguish between three classes
of RK methods. To facilitate their definition, let
\beq
X(\mu)=\left(I+\mu A\right)^{-1},
\eeq
where $\mu>0$ is a stability parameter and $I$ is the identity matrix.
Furthermore, let $e$ denote the vector of length $M$ with all entries equal to unity.

\subsubsection{Strong stability preserving (SSP) methods}

An explicit Runge-Kutta method is strong stability preserving (SSP) if
it is possible to write the stage approximations
$y_i^{(2)},\ldots,y_i^{(M)}$ and the final solutions $u_i^{\rm RK}$
 as convex combinations of forward
 Euler predictors $u_i^{(1)},\ldots,u_i^{(\tilde M)}$. For $\tilde M\in\{1,2,3,5\}$,
 the optimal explicit SSP-RK methods use
 \begin{subequations}
 \begin{align}
 u_i^{(1)}&=u^n,\\
 u_i^{(m)}&=u_i^{(m-1)}-\frac{\Delta t}{|K_i|}\sum_{j\in\mathcal N_i}|S_{ij}|
 H_{ij}^*(u_i^{(m-1)}) ,\qquad m=2,\ldots,\tilde M.
 \end{align}
\end{subequations}
If the space discretization \eqref{afcRK} is BP, then so is the resulting full
discretization. The maximum order of such SSP-RK time integrators
is 4 in the explicit case and 6 in the implicit case \cite{david_book}.
For comparison with the next class of methods, we note that SSP methods
satisfy the entrywise inequalities
\begin{align} \label{SSP-stages}
AX(\mu) \ge 0, \ \ \ \ AX(\mu)e \le e,
\end{align}
as well as
\begin{align} \label{SSP-update}
b^T X(\mu) \ge 0, \ \ \ \ b^T X(\mu) e \le 1,
\end{align}
with an SSP coefficient $\mu>0$ such that the fully discrete
scheme is BP for $\Delta t\le \mu \Delta t^{\rm FE}$ if the forward
Euler method is BP for $\Delta t\le\Delta t^{\rm FE}$; see
\cite{david_book} for details.

\subsubsection{Internal SSP methods}
A larger class of RK methods satisfy \eqref{SSP-stages} for some $\mu >0$
but may violate \eqref{SSP-update}. That is, all intermediate approximations
$u_i^{(1)},\ldots,u_i^{(\tilde M)}$ are BP for $\Delta t\le \mu
\Delta t^{\rm FE}$  but the weights of the final linear combination
$u_i^{\rm RK}=\sum_{m=1}^{\tilde M}\beta_mu_i^{(m)}$ are not necessarily positive.
The family of such internal SSP methods includes all SSP schemes, but
also many other methods, such as
extrapolation methods based on the explicit Euler method (see,
e.g., \cite[Section II.9]{hairerI}), whose implementation details
are presented in Algorithm~\ref{alg:extrap} and Table~1.
We use these methods in some numerical experiments in Section~\ref{sec:num}.
They have an advantage over SSP methods in that they can be constructed to have
any desired order of accuracy.  Since the 
final stage of these methods is not BP, we constrain it using a flux
limiter as described in Section \ref{sec:limiting_last}. 

\subsubsection{General RK methods}

A  high-order Runge-Kutta method with Butcher stages of the form
\eqref{rk-y1}--\eqref{rk-ym} and \eqref{rk-update} is generally
not BP. Hence,
flux limiting may be required in intermediate stages and/or in the
final stage.

\begin{algorithm}\caption{Explicit Euler extrapolation ({\bf Ex-Euler}) for
  $\td{u}{t}=F(u)$}
\label{alg:extrap}
\begin{algorithmic}

\State $y^{(1)} := u^n$
\State $m:=1$
\For{$s = 1 \to S$}
    \State $m := m+1$
    \State $y^{(m)} := u^n + \frac{\Delta t}{s}F(u^n)$
    \For{$k=2 \to s-1$}
        \State $m := m+1$
        \State $y^{(m)} := y^{(m-1)} + \frac{\Delta t}{s}F(y^{(m-1)})$
    \EndFor
\EndFor
\State $M:=m$
\State $u^{n+1} := u^n + \Delta t\sum_{m=1}^M b_m F(y^{(m)})$
\end{algorithmic}
\end{algorithm}

\begin{table}[h!]\label{extrapolation_weights}
  \vspace{0.75cm}
\begin{tabular}{c|c} \hline
Order & Weights $[b_1,\ldots,b_M]$\\ \hline
 2 & $[0, 1]$ \\
 3 & $[0, -2, 3/2, 3/2]$ \\
 4 & $[0, 2, -9/2, -9/2, 8/3, 8/3, 8/3]$ \\
 5 & $[0, -4/3, 27/4, 27/4, -32/3, -32/3, -32/3, 125/24, 125/24, 125/24, 125/24]$ \\
\hline
\end{tabular}
\caption{Weights for extrapolation methods, to be used in Algorithm \ref{alg:extrap}.}
\end{table}

\subsection{Flux limiting for the final RK stage}\label{sec:limiting_last}

If the employed time-stepping method is not SSP and/or the space discretization
\eqref{afcRK} uses the unlimited antidiffusive fluxes $F_{ij}^*=F_{ij}$, 
the bound-preserving low-order FE-LLF scheme
\beq\label{bp-update}
u_i^{\rm FE}=u_i^n-\frac{\Delta t}{|K_i|}\sum_{j\in\mathcal N_i}|S_{ij}|
H_{ij}^{\rm FE},\qquad H_{ij}^{\rm FE}=H(u_i^n,u_j^n,\mathbf{n}_{ij})
\eeq
can be combined with the final stage \eqref{rk-update} of the high-order
RK-LLF method
\beq
u_i^{\rm RK}=u_i^n-\frac{\Delta t}{|K_i|}\sum_{j\in\mathcal N_i}|S_{ij}|
H_{ij}^{\rm RK}=u_i^{\rm FE}+\frac{\Delta t}{|K_i|}\sum_{j\in\mathcal N_i}
F_{ij}^{\rm RK},\qquad  F_{ij}^{\rm RK}=
|S_{ij}|(H_{ij}^{\rm FE}-H_{ij}^{\rm RK})
\eeq
to produce a nonlinear flux-limited approximation of the form
\beq\label{fctgmc}
u_i^{n+1}=u_i^n-\frac{\Delta t}{|K_i|}\sum_{j\in\mathcal N_i}
|S_{ij}|\left[\alpha_{ij}H_{ij}^{\rm RK}+(1-\alpha_{ij})H_{ij}^{\rm FE}\right]
=u_i^{\rm FE}+\frac{\Delta t}{|K_i|}\sum_{j\in\mathcal N_i}
\alpha_{ij}F_{ij}^{\rm RK}.
\eeq
This fully discrete scheme reduces to \eqref{bp-update} for
$\alpha_{ij}=0$ and to \eqref{rk-update} for $\alpha_{ij}=1$. It is BP
w.r.t.\ the admissible set $\mathcal G_i=[u_i^{\min}, u_i^{\max}]$ if the 
definition of the correction factors $\alpha_{ij}$ guarantees that
\beq\label{BPfinal}
u_i^{\min}\le u_i^{n+1}\le u_i^{\max}.
\eeq

Zalesak's FCT method \cite{zalesak79} preserves the  BP property of 
 $u_i^{\rm FE}$
 using algorithm \eqref{p-def}--\eqref{alpha-def} to calculate
correction factors $\alpha_{ij}$ such that the limited fluxes
$F_{ij}^*=\alpha_{ij}F_{ij}^{\rm RK}$ satisfy
\beq Q_i^-=
\frac{|K_i|}{\Delta t}\left(u_i^{\min}-u_i^{\rm FE}\right)\le
\sum_{j\in\mathcal N_i}
F_{ij}^*\le
\frac{|K_i|}{\Delta t}\left(u_i^{\max}-u_i^{\rm FE}\right)=Q_i^+.\eeq
We remark that the temporal accuracy of classical
FCT schemes is restricted to second order. To our knowledge,
the first combination of FCT with high-order RK methods was
considered in \cite{CG-BFCT}.

As an alternative to FCT, we propose a flux limiter
that calculates $F_{ij}^*$ using the GMC bounds \eqref{bounds_glob}
in algorithm
\eqref{p-def}--\eqref{alpha-def}.
Invoking  \eqref{afc-coeff}, the flux-corrected scheme can
again be written as
$$
u_i^{n+1}
=u_i^n+(1+\gamma)\frac{\Delta t d_i}{|K_i|}\left(\bar u_i^{n}-u_i^n\right),
$$
where $\bar u_i^n\in[u_i^{\min}, u_i^{\max}]$. This  proves that
the fully discrete scheme is BP for time steps satisfying \eqref{cfl}.

\begin{rmk}
Update  \eqref{fctgmc} is mass conservative in the sense that  
$\sum_{i=1}^{N_h}|K_i|u_i^{n+1}=\sum_{i=1}^{N_h}|K_i|u_i^n$ because the
correction factors satisfy the symmetry condition
$\alpha_{ij}=\alpha_{ji}$ and the fluxes sum to zero.
\end{rmk}

\begin{rmk}\label{remark:fct_vs_gmc}
  Unlike FCT-like predictor-corrector approaches, the GMC limiter produces correction factors $\alpha_{ij}$ that do not depend on $\Delta t$. Consequently, this 
flux limiting strategy leads to well-defined nonlinear discrete problems and does not inhibit 
  convergence to steady-state solutions. 
\end{rmk}

\subsection{Flux limiting for intermediate RK stages}\label{sec:stagewise}

If the BP property needs to be enforced not only at the final stage
\eqref{rk-update} but also at intermediate stages \eqref{rk-ym}
of the Runge-Kutta method, flux-corrected approximations of the form
\beq\label{stage_butcher_sw}
y_i^{*,(m)}=
u_i^n-\frac{\Delta t}{|K_i|}\sum_{\jinN}\left[c_m|S_{ij}|H_{ij}^{\rm FE}
  -F_{ij}^{*,(m)}\right]=y_i^{{\rm FE},(m)}+
\frac{\Delta t}{|K_i|}\sum_{\jinN}F_{ij}^{*,(m)}
\eeq
can be calculated using 
\beq\label{stage_butcher_yL}
y_i^{{\rm FE},(m)}=
u_i^n-\frac{c_m\Delta t}{|K_i|}\sum_{\jinN}|S_{ij}|H_{ij}^{\rm FE},\qquad
F_{ij}^{*,(m)}=\alpha_{ij}^{(m)}F_{ij}^{(m)},
\eeq
 where
\beq\label{RK-stage_flux}
F_{ij}^{(m)}=c_m|S_{ij}|H_{ij}^{\rm FE}-\sum_{s=1}^{m-1}a_{ms}\int_{S_{ij}}
H_{ij}\big(\hat{y}^{(s)}\big)\ds, 
\eeq
is the antidiffusive flux. 
Here we used the simplified notation 
\begin{align}
  H_{ij}(z):=H(z_i(\bfx),z_j(\bfx),\mathbf{n}_{ij}).
\end{align}
The correction factors $\alpha_{ij}^{(m)}$ can
be calculated as in Section~\ref{sec:limiting_last}. In the GMC
version of the flux limiting procedure, the
bounds $Q_i^{\pm}$ should be multiplied by $c_m$.

\subsection{Slope limiting for high-order reconstructions}

Stagewise limiting guarantees the BP property of the
cell averages $y_i^{(s)}$. If the calculation of
$H_{ij}\left(\hat{y}^{(s)}\right)$
 requires the BP property of
the  high-order reconstructions
$\hat y_i^{(s)}:K_i\to\R$, it can
be enforced using slope limiting to blend a
reconstructed polynomial
$\hat y_i^{(s)}$
and a BP cell average as follows:
\begin{itemize}
\item If an intermediate cell
average $y_i$ is BP w.r.t. $[u_i^{\min},u_i^{\max}]$,
then there exists  $\theta_i\in[0,1]$ such that
\beq\label{limrec}
\hat y_i^*(\mathbf{x}_p):=y_i+
\theta_i\left(\hat y_i(\mathbf{x}_p)
-y_i\right)\in[u_i^{\min},u_i^{\max}]
\eeq
at each quadrature point $\mathbf{x}_p\in\partial K_i$ at
which the value of $\hat y_i^*(\mathbf{x}_p)$ is
required for calculation of $H_{ij}$. Adapting
the Barth-Jespersen formula \cite{barthjesp} to this
setting, we define the correction factor
\begin{equation}\label{slopelim}
\theta_i=\min_{p}
\left\{\begin{array}{ll}
\min\left\{1,\frac{u^{\max}_i-y_i}{\hat y_i(\mathbf{x}_p)-y_i}\right\} &
\mbox{if}\ \ \hat y_i(\mathbf{x}_p)>u^{\max}_i,\\
1 & \mbox{if}\ \ \hat y_i(\mathbf{x}_p)\in[u^{\min}_i, u^{\max}_i],\\
\min\left\{1,\frac{u^{\min}_i-y_i}{\hat y_i(\mathbf{x}_p)-y_i}\right\} &
\mbox{if}\ \ \hat y_i(\mathbf{x}_p)<u^{\min}_i.
\end{array}\right.
\end{equation}

We remark that slope limiters of this
kind were used, e.g., by Zhang and Shu \cite{ZhangShu2010,zhang2011maximum}
in the context of positivity-preserving WENO-DG schemes combined with
SSP-RK time discretizations.

\item
If the intermediate cell averages are not BP because
the RK method does not satisfy \eqref{SSP-stages}
or the flux limiter for $F_{ij}^{(m)}$
is deactivated, limited reconstructions of the form
 \beq\hat y_i^*(\mathbf{x}_p):=u_i^n+
 \theta_i(\hat y_i(\mathbf{x}_p)-u_i^n)\in[u_i^{\min},u_i^{\max}]
 \eeq
 may be employed. The conservation properties of the limited scheme
 are not affected by using a convex combination of $\hat y_i$ and
 $u_i^n$ for calculation of the high-order LLF fluxes $H_{ij}(\hat y_i^*)$.
\end{itemize}

\section{Case study: flux-limited WENO-RK schemes}\label{sec:BPschemes}

As mentioned in the introduction, the presented flux limiters can be applied to different kinds of high-order discretizations in space and time. 
In the numerical experiments of Section \ref{sec:num}, we use a fifth-order WENO spatial discretization 
combined with explicit high-order Runge-Kutta methods. The
corresponding sequence of solution updates is defined by
the Butcher tableau\smallskip

\begin{align}\label{butcher_tableau}
\begin{tabular}{c|ccccc}
  0        &          &          &         &            & \\
  $c_2$    & $a_{21}$ &          &         &            & \\
  $c_3$    & $a_{31}$ & $a_{32}$ &         &            & \\
  $\vdots$ & $\vdots$ & $\ddots$ &         &            & \\
  $c_M$    & $a_{M1}$ & $a_{M2}$ & $\dots$ & $a_{MM-1}$ & \\ \hline
           & $b_1$    & $b_2$    & $\dots$ & $b_{MM-1}$ & $b_M$.
\end{tabular}
\end{align}
\smallskip

To illustrate the use of limiting techniques for antidiffusive fluxes
depending on time and/or space
discretizations, we now detail specific combinations of high-order
RK methods with GMC-type flux limiters. A numerical study of the
following RK methods is performed in Section \ref{sec:num}:
\begin{itemize}
  \item SSP54: a fourth-order strong stability preserving method. 
  \item ExE-RK5: a fifth-order extrapolated Euler Runge-Kutta method. 
  \item RK76: a sixth-order Runge-Kutta method. 
\end{itemize}
The details of these time integration methods can be found in \ref{appendix:RK_methods}. The algorithms presented in \ref{appendix:one-dim_GMC_ExRK} define WENO-GMC discretizations for uniform grids in 1D. Since SSP time discretizations preserve the BP properties of flux-corrected semi-discrete schemes, we equip the SSP54 method with the spatial GMC limiter derived in Section~\ref{subsec:gmc}. For extrapolation methods, the application of this limiter is sufficient to guarantee that intermediate solutions are BP but the final stage needs to be constrained using \eqref{fctgmc}. For general RK methods, the space-time limiting techniques of Sections \ref{sec:limiting_last} and~\ref{sec:stagewise} can be used to enforce the BP property. In this context, a flux limiter can be invoked in each stage or only in the final stage (allowing bound-violating approximations in the intermediate stages).
We therefore test the following limiting strategies for the RK methods
under investigation:
\begin{itemize}
\item SSP54-GMC: SSP54 with GMC space limiting for $F_{ij}^{(m)}$ in
  each intermediate stage.
\item ExE-RK5-GMC: ExE-RK5 with GMC space limiting for
  $F_{ij}^{(m)}$ in each intermediate stage and GMC space-time limiting
for $F_{ij}^{\rm RK}$ in the final stage.
\item RK76-GMC: RK76 with GMC space-time limiting for $F_{ij}^{\rm RK}$
only in the final stage.
\item Sw-RK76-GMC: RK76  with GMC space-time limiting for $F_{ij}^{(m)}$
  in each intermediate stage and for $F_{ij}^{\rm RK}$ in the final stage.
\end{itemize}
The  SSP54-GMC and RK76-GMC methods are tested in all numerical experiments of Section \ref{sec:num}. To demonstrate that stagewise BP limiting does not degrade the high-order accuracy of the baseline scheme, at least if the bounds are global, we include the results of ExE-RK5-GMC and Sw-RK76-GMC convergence studies for test problems with smooth exact solutions. In the rest of the numerical experiments, the stagewise BP schemes perform similarly to SSP54-GMC and RK76-GMC. Therefore, no further ExE-RK5-GMC and Sw-RK76-GMC results are  reported in Section~\ref{sec:num}. 

\subsection{Summary of the fully discrete BP schemes}
For the reader's convenience, we now summarize the main steps
of the fully discrete BP 
algorithms. 

\subsubsection{Implementation of SSP54-GMC}\label{sec:ssp54-gmc}
The formulas for the stage approximations
$y_i^{(1)}, \ldots, y_i^{(5)}$ are presented in \ref{sec:ssp54}. 
The final stage result is given by $u_i^{n+1}=y_i^{(5)}$. For $k=1,\ldots,4$,
we need to evaluate the time derivatives
\begin{align}\label{RHS_GMC}
  F_i(y^{(k)}):=\frac{1}{|K_i|}\sum_{\jinN}|S_{ij}|\lambda_{ij}\big(y^{(k)}\big)\big(\bar u_{ij}^*\big(y^{(k)}\big)-y_i^{(k)}\big),
\end{align}
where $\lambda_{ij}$ is an upper bound for the wave speed
of the Riemann problem associated with face $S_{ij}$. 
In Section \ref{sec:num}, we specify the values of $\lambda_{ij}$ that 
we use for the different numerical experiments. For $u\in
\{y_i^{(1)}, \ldots, y_i^{(5)}\}$, we calculate the 
bar states $\bar u_{ij}^*(u)$  as follows:
\begin{enumerate}[i.]
\item Compute the low-order bar states $\bar u_{ij}^L$ using  \eqref{bar_states}.
  \item Compute the antidiffusive fluxes $F_{ij}$ using \eqref{Fijdef}. 
  \item Compute the GMC bounds $Q_i^{\pm}$ defined by \eqref{bounds_glob}. 
  \item Compute the correction factors  $\alpha_{ij}$ using
  \eqref{p-def}-\eqref{alpha-def}.
  \item Compute the limited antidiffusive fluxes $F_{ij}^*=\alpha_{ij}F_{ij}$.
  \item Compute the flux-corrected bar states $\bar u_{ij}^*$ using \eqref{barstar}.
\end{enumerate}

\subsubsection{Implementation of ExE-RK5-GMC}\label{sec:ExE-RK5-GMC}
The ExE-RK5 method that we use is defined in \ref{appendix:ExE-RK5}. The
BP stage approximations $y_i^{(1)},\ldots, y_i^{(11)}$ are
calculated as in Section \ref{sec:ssp54-gmc}. In the final stage,
we update $u=u^n$ as follows: 
\begin{enumerate}[i.]
\item Compute the low-order bar states $\bar u_{ij}^L$ using \eqref{bar_states}.
\item Compute the high-order fluxes $H_{ij}^{\rm RK}$ using \eqref{rk-update}.
\item Compute the low-order fluxes $H_{ij}^{\rm FE}=H(u_i^n,u_j^n,\mathbf{n}_{ij})$.
\item Compute the antidiffusive fluxes $F_{ij}^{\rm RK}=
  |S_{ij}|(H_{ij}^{\rm FE}-H_{ij}^{\rm RK})$.
   \item Compute the low-order predictor $u_i^{\rm FE}$ using \eqref{bp-update}.
\item Compute the GMC bounds $Q_i^{\pm}$ defined by \eqref{bounds_glob}.
\item Compute the correction factors $\alpha_{ij}$ using 
 \eqref{p-def}-\eqref{alpha-def}.
\item Use $F_{ij}^*=\alpha_{ij}F_{ij}^{\rm RK}$ to
  compute $u_i^{n+1}$ defined by \eqref{fctgmc}.
\end{enumerate}

\subsubsection{Implementation of RK76-GMC}
The  Butcher tableau of the RK76 method is presented in \ref{sec:RK76}. The
stage approximations $y_i^{(1)},\ldots,y_i^{(7)}$ are calculated using
\eqref{stage-butcher} without applying any flux limiter. The final
RK update is constrained using the GMC space-time limiter, i.e., 
the algorithm presented in Section \ref{sec:ExE-RK5-GMC}.

\subsubsection{Implementation of Sw-RK76-GMC}
To constrain the intermediate Butcher stages of the RK76 method, we proceed as follows: 
\begin{enumerate}[i.]
\item Compute the low-order predictor $y_i^{{\rm FE},(m)}$
  using \eqref{stage_butcher_yL}.
\item Compute the antidiffusive flux 
  $F_{ij}^{(m)}$ using \eqref{RK-stage_flux}.
\item Compute the GMC bounds $Q_i^{\pm}$ defined by \eqref{bounds_glob}.
\item Compute the correction factors $\alpha_{ij}^{(m)}$ using 
 \eqref{p-def}-\eqref{alpha-def}.
  \item Use $F_{ij}^{*,(m)}=\alpha_{ij}^{(m)}F_{ij}^{(m)}$ to
    compute $y_i^{*,(m)}$ defined by \eqref{stage_butcher_sw}.
\end{enumerate}
To guarantee that the final RK update is BP, we again use the
algorithm presented in Section \ref{sec:ExE-RK5-GMC}.

\section{Numerical examples}\label{sec:num}
In this section, we perform a series of numerical experiments to study the properties of the methods selected in Section \ref{sec:BPschemes}. We begin with an accuracy test for WENO and WENO-GMC space discretizations. Then we test the flux-limited WENO-RK methods for time-dependent problems.

Unless mentioned otherwise, we use a fifth-order WENO \cite{Shu2009} spatial discretization on a uniform mesh of $N_h$ one-dimensional cells
$K_i=[x_{i-1/2},x_{i+1/2}]$ with 
$|K_i|=\Delta x$. For time-dependent problems, the default
choice of the time step is $\Delta t=0.4\Delta x/(1+\gamma)$,
where $\gamma\geq 0$ is the parameter of the 
limiting constraints \eqref{BPconstr2}. In all experiments,
we enforce the global bounds of the initial data $u(x,0)$, i.e., use
\begin{align*}
  u_i^{\min}&=u^{\min} := \min_x u(x,0), \\
  u_i^{\max}&=u^{\max} := \max_x u(x,0)
\end{align*}
for $i=1, \dots, N_h$. To quantify the magnitude of undershoots and overshoots
(if any), we report 
\beq
\delta = \min\{\delta^-, \delta^+\},
\eeq
where
\begin{align*}
  \delta^{-} = \min_t ~\min_{i=1, \dots, N_h} ~u_i(t) - u^{\min}, \qquad
  \delta^{+} = \min_t ~\min_{i=1, \dots, N_h} ~u^{\max}-u_i(t).
\end{align*}
Note that $\delta\geq 0$ for any BP numerical solution. In practice, the
low-order and flux-limited methods are positivity preserving to
machine precision. Therefore, a numerical value of $\delta$ may be
a small negative number. Since conservation laws are also satisfied
to machine precision, it is acceptable to simply clip the solution. 
If the exact solution is available, we calculate the discrete $L_1$ error 
\begin{align*}
  E_1(t) &= \Delta x \sum_{i=1}^{N_h}|\tilde{u}_i(t)-u^{\text{exact}}(x_i,t)|
\end{align*}
and the corresponding Experimental Order of Convergence (EOC) using
the following fifth-order polynomial reconstruction of the
numerical solution at the center $x_i$ of the cell:
\begin{align*}
  \tilde{u}_i
  = \frac{1}{1920} \left(9 u_{i-2}-116 u_{i-1} + 2134 u_i - 116 u_{i+1} + 9 u_{i+2}\right).
\end{align*}

\subsection{Convergence of a WENO-GMC semi-discretization}\label{sec:num_semi_discrete}
In this section, we test the convergence properties of the spatial semi-discretization 
using the GMC limiters from Section \ref{sec:gmcl}. 
To this end, we consider the one-dimensional conservation law
\beq\label{cons_law_semi}
\frac{\partial u}{\partial t} + \frac{\partial f(u)}{\partial x} = 0 \quad \mbox{ in }  \quad\Omega=(0,1),
\eeq
where $f(u)$ is the flux function. The time derivative of the exact cell average is given by 
\beq\label{WENO1Dex}
\td{u_i}{t}=-\frac{1}{\Delta x}\int_{x_{i-1/2}}^{x_{i+1/2}}\frac{\partial f(u)}{\partial x}
\mathrm{d}x
=-\frac{f(u(x_{i+1/2}))-f(u(x_{i-1/2}))}{\Delta x}.
\eeq
We discretize \eqref{cons_law_semi} in space using a fifth-order WENO scheme
which approximates \eqref{WENO1Dex} by
\beq\label{WENO1Dtest}
\td{u_i^\WENO}{t}=-\frac{H\left(\hat u_i^+,\hat u_{i+1}^-,1\right)
-H\left(\hat u_{i-1}^+,\hat u_i^-,1\right)}{\Delta x}.
\eeq
The Lax-Friedrichs flux
$H(\cdot,\cdot,\cdot)$ is defined using
$\lambda_{i+1/2} = \max\{u_i,u_{i+1},\hat u_i^+,\hat u_{i+1}^-\}
$ in this test.
The interface values $\hat u_i^-=\hat u_i(x_{i-1/2})$ and $\hat u_i^+=\hat u_i(x_{i+1/2})$ are determined by evaluating the WENO polynomial reconstruction $\hat u_i\in\mathbb{P}_5(K_i)$ at the cell interfaces $x_{i\pm1/2}$.

Applying the GMC flux limiter to \eqref{WENO1Dtest},  we obtain the WENO-GMC semi-discretization
\beq
\td{u_i^\GMC}{t}=(\lambda_{i+1/2}+\lambda_{i-1/2})\frac{\bar u_i^*-u_i}{\Delta x}.
\eeq

To assess its spatial accuracy, 
we consider $u(x)=\exp\left(-100(x-0.5)^2\right)$ and the nonlinear flux function $f(u)=u^2/2$. 
In Table \ref{table:semi}, we report the discrete $L^1$ errors
\begin{align*}
E_1^{\WENO}&=\Delta x\sum_{i=1}^{N_h}\left|\td{u_i}{t}-\td{u_i^{\WENO}}{t}\right|,\\
E_1^{\GMC}&=\Delta x\sum_{i=1}^{N_h}\left|\td{u_i}{t}-\td{u_i^{\GMC}}{t}\right|,
\end{align*}
as well as the EOCs of the WENO and WENO-GMC semi-discretizations. 
We consider multiple values of the GMC parameter $\gamma\geq 0$ and achieve optimal convergence rates with $\gamma\ge 0.5$.

\begin{table}[!ht]\small
  \begin{center}
    \begin{tabular}{||c||c|c||c|c||c|c||c|c||} \hline
      $N_h$ & $E_1^{\WENO}$ & EOC & 
      $E_1^{\GMC}$, $\gamma=0$ & EOC & $E_1^{\GMC}$, $\gamma=0.5$ & EOC  & $E_1^{\GMC}$, $\gamma=1$ & EOC 
      \\ \hline
      25   & 1.35e-03 &   -- & 1.35e-03 &   -- & 1.35e-03 &   -- & 1.35e-03 &   -- \\
      50   & 6.82e-05 & 4.30 & 5.12e-04 & 1.40 & 6.82e-05 & 4.30 & 6.82e-05 & 4.30 \\
      100  & 1.04e-06 & 6.04 & 6.60e-05 & 2.95 & 1.04e-06 & 6.04 & 1.04e-06 & 6.04 \\
      200  & 1.53e-08 & 6.08 & 8.30e-06 & 2.99 & 1.53e-08 & 6.08 & 1.53e-08 & 6.08 \\
      400  & 2.29e-10 & 6.06 & 1.04e-06 & 3.00 & 2.29e-10 & 6.06 & 2.29e-10 & 6.06 \\
      800  & 3.48e-12 & 6.04 & 1.30e-07 & 3.00 & 3.48e-12 & 6.04 & 3.48e-12 & 6.04 \\
      1600 & 5.36e-14 & 6.02 & 1.63e-08 & 3.00 & 5.36e-14 & 6.02 & 5.36e-14 & 6.02 \\ \hline
    \end{tabular}
    \caption{Grid convergence study for the 5th-order WENO and
      WENO-GMC discretizations of $\pd{}{x}(u^2/2)$. \label{table:semi}}
  \end{center}
\end{table}

\subsection{Linear advection}\label{sec:num_linear_advection}
The first test problem for our numerical study of the flux-limited space-time discretizations defined in Section \ref{sec:BPschemes} is 
the one-dimensional linear advection equation 
\beq\label{lin_adv}
\pd{u}{t}+a\pd{u}{x}=0 \quad\mbox{ in }\quad\Omega=(0,1) 
\eeq
with constant velocity $a=1$. 
The initial condition is given by the smooth function 
\begin{align}\label{lin_adv_smooth}
  u(x,0)=\exp(-100(x-0.5)^2).
\end{align}
We solve \eqref{lin_adv} up to the final time $t=1$ using
 $\lambda_{i+1/2}=1$ for all $i$.
The results of a grid convergence study for SSP54-GMC,
ExE-RK5-GMC, RK76-GMC, Sw-RK76-GMC, and the underlying
WENO-RK discretizations
are shown in Table \ref{table:lin_adv_conv}. The negative
values of $\delta$ indicate that the high-order
baseline schemes may, indeed,
produce undershoots or overshoots on coarse meshes.

In this test, all GMC-constrained 
 schemes deliver optimal convergence
 rates for $\gamma=1$. The only scheme that preserves the
 full accuracy for $\gamma=0$ is RK76-GMC, the method
 which performs flux limiting just once per time step. We
 remark that even this least dissipative method requires
 the use of $\gamma>0$ to achieve optimal EOCs for other
 test problems that we consider below.
 
\begin{table}[!ht]\scriptsize
  \begin{center}
    \subfloat[SSP54 and SSP54-GMC]{
      \begin{tabular}{||c||c|c|c||c|c|c||c|c|c||} \hline
        & \multicolumn{3}{c||}{baseline}  
        &\multicolumn{3}{c||}{GMC, $\gamma=0$} 
        &\multicolumn{3}{c||}{GMC, $\gamma=1$} \\ \cline{2-10}
        $N_h$ & $E_1$ & EOC & $\delta$ & $E_1$ & EOC & $\delta$ & $E_1$ & EOC & $\delta$ \\ \hline
        25   & 2.43e-02 &  --  & -2.00e-05 & 2.43e-02 &  --  & 1.28e-10 & 2.43e-02 &  --  & 7.58e-11 \\
        50   & 2.30e-03 & 3.40 & -3.26e-08 & 2.41e-03 & 3.34 & 2.03e-11 & 2.29e-03 & 3.40 & 4.95e-12 \\
        100  & 1.22e-04 & 4.24 & -6.45e-11 & 1.37e-04 & 4.13 & 5.64e-12 & 1.22e-04 & 4.23 & 1.07e-12 \\
        200  & 4.22e-06 & 4.85 &  1.65e-11 & 1.35e-05 & 3.34 & 1.65e-11 & 4.22e-06 & 4.85 & 1.65e-11 \\
        400  & 1.35e-07 & 4.97 &  1.51e-11 & 1.89e-06 & 2.84 & 1.51e-11 & 1.35e-07 & 4.97 & 1.51e-11 \\
        800  & 4.24e-09 & 4.99 &  1.45e-11 & 2.89e-07 & 2.71 & 1.45e-11 & 4.24e-09 & 4.99 & 1.45e-11 \\
        1600 & 2.17e-10 & 4.29 &  1.42e-11 & 4.48e-08 & 2.69 & 1.42e-11 & 2.15e-10 & 4.30 & 1.42e-11 \\ \hline
      \end{tabular}
    }

    \subfloat[ExE-RK5 and ExE-RK5-GMC]{
      \begin{tabular}{||c||c|c|c||c|c|c||c|c|c||} \hline
        & \multicolumn{3}{c||}{baseline}  
        &\multicolumn{3}{c||}{GMC, $\gamma=0$} 
        &\multicolumn{3}{c||}{GMC, $\gamma=1$} \\ \cline{2-10}
        $N_h$ & $E_1$ & EOC & $\delta$ & $E_1$ & EOC & $\delta$ & $E_1$ & EOC & $\delta$ \\ \hline
        25   & 2.43e-02 &  --  & -2.00e-05 & 2.43e-02 &   -- & 1.23e-10 & 2.43e-02 &   -- & 1.51e-11 \\
        50   & 2.29e-03 & 3.40 & -3.26e-08 & 2.37e-03 & 3.35 & 1.95e-11 & 2.29e-03 & 3.40 & 4.91e-12 \\
        100  & 1.22e-04 & 4.23 & -6.47e-11 & 1.33e-04 & 4.16 & 5.51e-12 & 1.22e-04 & 4.23 & 6.82e-13 \\
        200  & 4.22e-06 & 4.85 &  1.65e-11 & 1.05e-05 & 3.66 & 1.65e-11 & 4.22e-06 & 4.85 & 1.65e-11 \\
        400  & 1.35e-07 & 4.97 &  1.51e-11 & 1.50e-06 & 2.80 & 1.51e-11 & 1.35e-07 & 4.97 & 1.51e-11 \\
        800  & 4.23e-09 & 4.99 &  1.45e-11 & 2.41e-07 & 2.64 & 1.45e-11 & 4.24e-09 & 4.99 & 1.45e-11 \\
        1600 & 1.33e-10 & 5.00 &  1.42e-11 & 3.83e-08 & 2.66 & 1.42e-11 & 1.33e-10 & 5.00 & 1.42e-11 \\ \hline
      \end{tabular}
    }

    \subfloat[RK76 and RK76-GMC]{
      \begin{tabular}{||c||c|c|c||c|c|c||c|c|c||} \hline
        & \multicolumn{3}{c||}{baseline}  
        &\multicolumn{3}{c||}{GMC, $\gamma=0$} 
        &\multicolumn{3}{c||}{GMC, $\gamma=1$} \\ \cline{2-10}
        $N_h$ & $E_1$ & EOC & $\delta$ & $E_1$ & EOC & $\delta$ & $E_1$ & EOC & $\delta$ \\ \hline
        25   & 2.43e-02 &   -- & -2.00e-05 & 2.43e-02 &  --  & 3.37e-11 & 2.43e-02 &   -- & 6.73e-12 \\
        50   & 2.29e-03 & 3.40 & -3.26e-08 & 2.29e-03 & 3.40 & 4.73e-12 & 2.29e-03 & 3.40 & 4.04e-13 \\
        100  & 1.22e-04 & 4.23 & -6.48e-11 & 1.22e-04 & 4.23 & 7.03e-13 & 1.22e-04 & 4.23 & 1.00e-13 \\
        200  & 4.22e-06 & 4.85 &  1.65e-11 & 4.22e-06 & 4.85 & 1.65e-11 & 4.22e-06 & 4.85 & 1.65e-11 \\
        400  & 1.35e-07 & 4.97 &  1.51e-11 & 1.35e-07 & 4.97 & 1.51e-11 & 1.35e-07 & 4.97 & 1.51e-11 \\
        800  & 4.23e-09 & 4.99 &  1.45e-11 & 4.23e-09 & 4.99 & 1.45e-11 & 4.24e-09 & 4.99 & 1.45e-11 \\
        1600 & 1.32e-10 & 5.00 &  1.42e-11 & 1.32e-10 & 5.00 & 1.42e-11 & 1.33e-10 & 5.00 & 1.42e-11 \\ \hline
      \end{tabular}
    }

    \subfloat[Sw-RK76-GMC]{
      \begin{tabular}{||c||c|c|c||c|c|c||} \hline
        &\multicolumn{3}{c||}{GMC, $\gamma=0$} 
        &\multicolumn{3}{c||}{GMC, $\gamma=1$} \\ \cline{2-7}
        $N_h$ & $E_1$ & EOC & $\delta$ & $E_1$ & EOC & $\delta$ \\ \hline
        25   & 2.43e-02 &  --  & 3.37e-11 & 2.43e-02 &  --  & 6.73e-12 \\
        50   & 2.30e-03 & 3.40 & 4.79e-12 & 2.29e-03 & 3.40 & 4.28e-13 \\
        100  & 1.22e-04 & 4.24 & 6.25e-13 & 1.22e-04 & 4.23 & 1.24e-13 \\
        200  & 5.40e-06 & 4.50 & 1.65e-11 & 4.22e-06 & 4.85 & 1.65e-11 \\
        400  & 5.86e-07 & 3.20 & 1.51e-11 & 1.35e-07 & 4.97 & 1.51e-11 \\
        800  & 8.37e-08 & 2.81 & 1.45e-11 & 4.24e-09 & 4.99 & 1.45e-11 \\
        1600 & 1.29e-08 & 2.70 & 1.42e-11 & 1.33e-10 & 5.00 & 1.42e-11 \\ \hline
      \end{tabular}
    }
    \caption{Grid convergence study for the linear advection problem \eqref{lin_adv} 
      with smooth initial data \eqref{lin_adv_smooth}.\label{table:lin_adv_conv}}
  \end{center}
\end{table}

To check how well a given scheme can preserve smooth peaks and capture 
discontinuities, we run the linear advection test with the initial data \cite{Guermond2011}
\beq\label{lin_adv_nsmooth}
u(x,0) = 
\begin{cases}
  \exp{(-300(2x-0.3)^2)} & \mbox{ if } |2x-0.3|\leq 0.25, \\
  1                  & \mbox{ if } |2x-0.9|\leq 0.2, \\
  \sqrt{1-\left(\frac{2x-1.6}{0.2}\right)^2} & \mbox{ if } |2x-1.6|\leq 0.2, \\
  0 & \mbox{\text{otherwise}}.
\end{cases}
\eeq
Figure \ref{fig:lin_adv_nsmooth} shows the results produced by SSP54-GMC,
RK76-GMC, and the corresponding baseline schemes at $t=1$ and $t=100$. Both
GMC schemes use $\gamma=1$ and preserve the global bounds without changing
the high-order WENO-RK approximation
in smooth regions. The deactivation of GMC
limiters leads to visible violations of the bounds in proximity to
steep gradients. The values of $\delta$ listed above the diagrams
quantify the amount of undershooting/overshooting for each scheme. 

\begin{figure}[!h]
  \begin{center}
  {\scriptsize
    \subfloat[SSP54 and SSP54-GMC at $t=1$]{    
      \begin{tabular}{c}
        baseline: $\delta=-4.97\times 10^{-6}$ \\
        GMC: $\delta=-1.11\times 10^{-15}$ \\
        \includegraphics[scale=0.4]{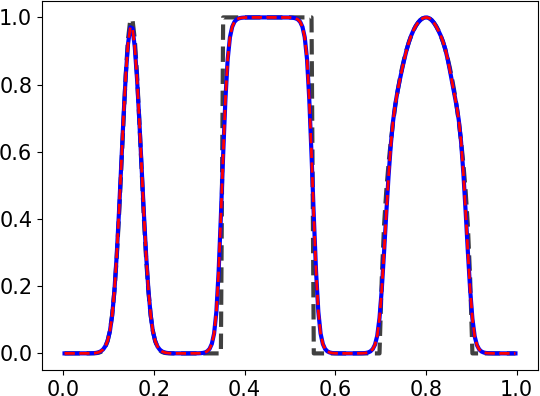}
      \end{tabular}
    }
    \quad
    \subfloat[SSP54 and SSP54-GMC at $t=100$]{    
      \begin{tabular}{c}
        baseline: $\delta=-1.32\times 10^{-2}$ \\
        GMC: $\delta=-4.44\times 10^{-16}$ \\
        \includegraphics[scale=0.4]{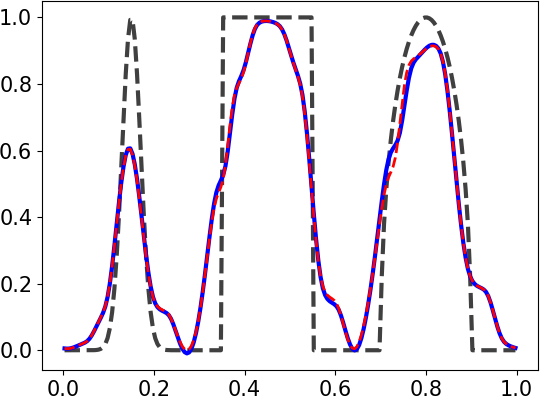}
      \end{tabular}
    }

    \vspace{5pt}
    \subfloat[RK76 and RK76-GMC at $t=1$]{    
      \begin{tabular}{c}
        baseline: $\delta=-4.97\times 10^{-6}$ \\
        GMC: $\delta=-1.11\times 10^{-15}$ \\
        \includegraphics[scale=0.4]{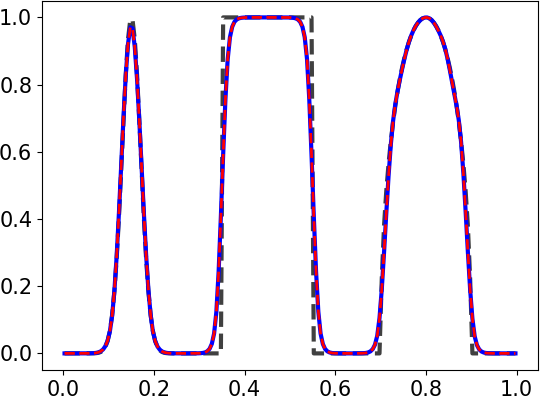}
      \end{tabular}
    }
    \quad
    \subfloat[RK76 and RK76-GMC at $t=100$]{    
      \begin{tabular}{c}
        baseline: $\delta=-1.32\times 10^{-2}$ \\
        GMC: $\delta=-1.11\times 10^{-15}$ \\
        \includegraphics[scale=0.4]{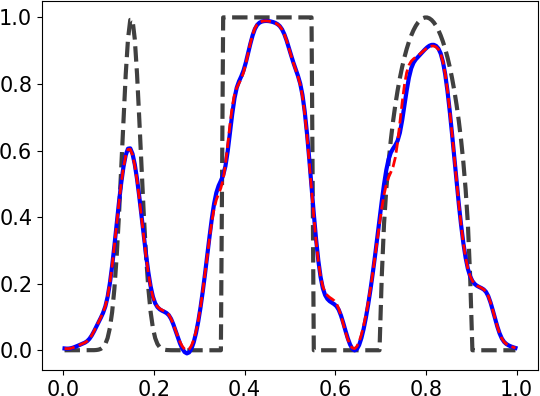}
      \end{tabular}
    }
  }
  \end{center}
  \caption{Numerical solutions to the linear advection problem
    \eqref{lin_adv} 
    with initial data \eqref{lin_adv_nsmooth}. Computations are
    performed using $N_h=200$ cells.
    The SSP54-GMC and RK76-GMC results are
    shown as dashed red lines. The
    results obtained with the baseline SSP54 and RK76
    schemes   are shown as solid blue lines. 
    \label{fig:lin_adv_nsmooth}}
\end{figure}

\subsection{Burgers equation}
To study the numerical behavior of the methods under investigation
in the context of nonlinear hyperbolic problems, we consider
the one-dimensional inviscid Burgers equation 
\beq\label{1Dburgers}
\pd{u}{t}+\pd{}{x}\left(\frac{u^2}{2}\right)=0 \quad\mbox{ in }\quad\Omega=(0,2\pi).
\eeq
Following Kurganov and Tadmor \cite{kurganov}, we 
use the smooth initial condition 
\beq\label{1Dburgers_init}
u(x,0) = 0.5 +\sin x. 
\eeq
The entropy solution of the initial value problem develops
a shock at the critical time $T_c=1$.  For $t<T_c$, the smooth
exact solution is defined by the nonlinear equation
$u(x,t) = 0.5+\sin(x-u(x,t)t)$, which can be derived
using the method of characteristics.

For this problem, we define the Lax-Friedrichs fluxes using
$
\lambda_{i+1/2} = \max\{u_i,u_{i+1},\hat u_i^+,\hat u_{i+1}^-\}
$.
The results of a grid convergence study 
are summarized in Table \ref{table:burgers_conv}. 
The errors and convergence rates correspond to the pre-shock time $T=0.5$. 
None of the flux-limited schemes converges optimally for $\gamma=0$. 
Using $\gamma=1$, we recover full accuracy with SSP54-GMC, RK76-GMC,
and Sw-RK76-GMC. The ExE-RK5-GMC version delivers optimal convergence
rates for $\gamma\ge 2$. 

To study the ability of our schemes to capture the shock that forms
at $t=T_c$, we ran simulations up to the post-shock time $T=2$. As
shown in Fig. \ref{fig:burgers}, the SSP54-GMC and RK76-GMC
results coincide with the globally BP solutions produced by
the corresponding baseline schemes.

\begin{table}[!ht]\scriptsize
  \begin{center}
    \subfloat[SSP54 and SSP54-GMC]{
      \begin{tabular}{||c||c|c|c||c|c|c||c|c|c||} \hline
        & \multicolumn{3}{c||}{baseline}  
        &\multicolumn{3}{c||}{GMC, $\gamma=0$} 
        &\multicolumn{3}{c||}{GMC, $\gamma=1$} \\ \cline{2-10}
        $N_h$ & $E_1$ & EOC & $\delta$ & $E_1$ & EOC & $\delta$ & $E_1$ & EOC & $\delta$ \\ \hline
25  & 2.01e-03 &   -- & 2.72e-03 & 5.90e-03 & 0.00 & 3.17e-03 & 2.08e-03 & 0.00 & 2.70e-03 \\
50  & 1.12e-04 & 4.16 & 6.62e-04 & 7.51e-04 & 2.97 & 9.67e-04 & 1.16e-04 & 4.17 & 6.62e-04 \\
100 & 4.70e-06 & 4.58 & 1.84e-04 & 1.13e-04 & 2.73 & 2.75e-04 & 4.81e-06 & 4.59 & 1.64e-04 \\
200 & 2.12e-07 & 4.47 & 4.60e-05 & 1.62e-05 & 2.80 & 6.89e-05 & 2.16e-07 & 4.48 & 4.11e-05 \\
400 & 1.05e-08 & 4.34 & 1.15e-05 & 2.40e-06 & 2.76 & 1.72e-05 & 1.07e-08 & 4.34 & 1.03e-05 \\
800 & 6.29e-10 & 4.06 & 2.58e-06 & 3.68e-07 & 2.70 & 4.31e-06 & 6.16e-10 & 4.11 & 2.57e-06 \\ \hline
      \end{tabular}
    }

    \subfloat[ExE-RK5 and ExE-RK5-GMC]{
      \begin{tabular}{||c||c|c|c||c|c|c||c|c|c||c|c|c||} \hline
        & \multicolumn{3}{c||}{baseline}  
        &\multicolumn{3}{c||}{GMC, $\gamma=0$} 
        &\multicolumn{3}{c||}{GMC, $\gamma=1$} 
        &\multicolumn{3}{c||}{GMC, $\gamma=2$} \\ \cline{2-13}
        $N_h$ 
        & $E_1$ & EOC & $\delta$ 
        & $E_1$ & EOC & $\delta$ 
        & $E_1$ & EOC & $\delta$ 
        & $E_1$ & EOC & $\delta$ \\ \hline
25  & 2.04e-03 &   -- & 2.72e-03 & 6.66e-03 &   -- & 6.76e-04 & 2.08e-03 &   -- & 2.70e-03 & 2.10e-03 &   -- & 2.65e-03 \\
50  & 1.14e-04 & 4.16 & 6.62e-04 & 5.69e-04 & 3.55 & 9.68e-04 & 1.14e-04 & 4.20 & 6.62e-04 & 1.16e-04 & 4.17 & 6.59e-04 \\
100 & 4.79e-06 & 4.57 & 1.84e-04 & 1.22e-04 & 2.22 & 6.73e-05 & 4.26e-06 & 4.74 & 1.64e-04 & 4.82e-06 & 4.59 & 1.67e-04 \\
200 & 2.16e-07 & 4.47 & 4.60e-05 & 1.81e-05 & 2.75 & 1.51e-05 & 3.42e-07 & 3.64 & 4.01e-05 & 2.16e-07 & 4.48 & 4.17e-05 \\
400 & 1.06e-08 & 4.34 & 1.15e-05 & 2.57e-06 & 2.82 & 3.67e-06 & 3.47e-08 & 3.30 & 1.00e-05 & 1.06e-08 & 4.35 & 1.04e-05 \\
800 & 5.62e-10 & 4.24 & 2.58e-06 & 3.63e-07 & 2.83 & 9.11e-07 & 5.33e-09 & 2.70 & 2.43e-06 & 5.62e-10 & 4.24 & 2.58e-06 \\ \hline
      \end{tabular}
    }

    \subfloat[RK76 and RK76-GMC]{
      \begin{tabular}{||c||c|c|c||c|c|c||c|c|c||} \hline
        & \multicolumn{3}{c||}{baseline}  
        &\multicolumn{3}{c||}{GMC, $\gamma=0$} 
        &\multicolumn{3}{c||}{GMC, $\gamma=1$} \\ \cline{2-10}
        $N_h$ & $E_1$ & EOC & $\delta$ & $E_1$ & EOC & $\delta$ & $E_1$ & EOC & $\delta$ \\ \hline
25  & 2.04e-03 &   -- & 2.72e-03 & 2.63e-03 &   -- & 2.77e-03 & 2.08e-03 &   -- & 2.70e-03 \\
50  & 1.14e-04 & 4.16 & 6.62e-04 & 2.05e-04 & 3.68 & 6.69e-04 & 1.16e-04 & 4.17 & 6.62e-04 \\
100 & 4.79e-06 & 4.57 & 1.84e-04 & 1.95e-05 & 3.40 & 1.84e-04 & 4.82e-06 & 4.59 & 1.64e-04 \\
200 & 2.16e-07 & 4.47 & 4.60e-05 & 2.48e-06 & 2.98 & 4.60e-05 & 2.16e-07 & 4.48 & 4.11e-05 \\
400 & 1.06e-08 & 4.34 & 1.15e-05 & 3.66e-07 & 2.76 & 1.15e-05 & 1.06e-08 & 4.35 & 1.03e-05 \\
800 & 5.62e-10 & 4.24 & 2.58e-06 & 5.61e-08 & 2.71 & 2.58e-06 & 5.62e-10 & 4.24 & 2.57e-06 \\ \hline
      \end{tabular}
    }

    \subfloat[Sw-RK76-GMC]{
      \begin{tabular}{||c||c|c|c||c|c|c||} \hline
        &\multicolumn{3}{c||}{GMC, $\gamma=0$} 
        &\multicolumn{3}{c||}{GMC, $\gamma=1$} \\ \cline{2-7}
        $N_h$ & $E_1$ & EOC & $\delta$ & $E_1$ & EOC & $\delta$ \\ \hline
25  & 2.64e-03 &   -- & 2.68e-03 & 2.08e-03 &   -- & 2.70e-03 \\
50  & 2.39e-04 & 3.47 & 6.43e-04 & 1.16e-04 & 4.17 & 6.62e-04 \\ 
100 & 2.58e-05 & 3.21 & 1.75e-04 & 4.82e-06 & 4.59 & 1.64e-04 \\
200 & 3.81e-06 & 2.76 & 4.37e-05 & 2.16e-07 & 4.48 & 4.11e-05 \\
400 & 5.91e-07 & 2.69 & 1.09e-05 & 1.06e-08 & 4.35 & 1.03e-05 \\
800 & 8.91e-08 & 2.73 & 2.73e-06 & 5.62e-10 & 4.24 & 2.57e-06 \\ \hline
      \end{tabular}
    }
    \caption{Grid convergence study for the Burgers equation \eqref{1Dburgers} with smooth initial data 
      \eqref{1Dburgers_init}.\label{table:burgers_conv}}
  \end{center}
\end{table}

\begin{figure}[!h]
  \begin{center}
    {\scriptsize
      \subfloat[SSP54 and SSP54-GMC]{    
        \begin{tabular}{c}
          baseline: $\delta=1.64\times 10^{-4}$ \\
          GMC: $\delta=1.64\times 10^{-4}$ \\
          \includegraphics[scale=0.4]{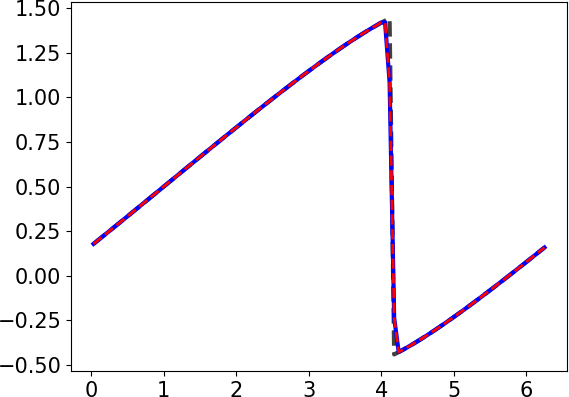}
        \end{tabular}
      }
      \quad
      \subfloat[RK76 and RK76-GMC]{
        \begin{tabular}{c}
          baseline: $\delta=1.64\times 10^{-4}$ \\
          GMC: $\delta=1.64\times 10^{-4}$ \\
          \includegraphics[scale=0.4]{figures/nonlin_ssp_gmc.png}
        \end{tabular}
      }
    }
  \end{center}
  \caption{Nonlinear Burgers problem \eqref{1Dburgers} 
    with smooth initial data \eqref{1Dburgers_init}.
    The dashed gray lines depict the exact solution at $t=2$. The  
    WENO-RK solutions calculated on a uniform mesh of
    $N_h=100$ cells without and with GMC limiting are shown
    as dashed red and solid blue lines, respectively. 
    \label{fig:burgers}}
\end{figure}

\subsection{One-dimensional KPP problem}
In the last test, we solve the one-dimensional KPP problem \cite{kpp}.
It equips the conservation law 
\begin{subequations}\label{kpp}
\beq
\pd{u}{t}+\pd{f(u)}{x}=0 \quad \mbox{ in } \quad \Omega=(0,1)
\eeq
with the nonlinear and nonconvex flux function
\beq
f(u)=
\begin{cases}
  \frac{1}{4}u(1-u) &\mbox{ if } u< \frac{1}{2}, \\
  \frac{1}{2}u(u-1)+\frac{3}{16} &\mbox{ if } \frac{1}{2}\leq u.
\end{cases}
\eeq
\end{subequations}
The initial condition is given by \cite{ern}
\beq\label{kpp_init}
u(x,0)=
\begin{cases}
  0 &\mbox{ if } x\in[0,0.35], \\
  1 &\mbox{ if } x\in(0.35,1].
\end{cases}
\eeq
We define the Lax-Friedrichs fluxes using $\lambda_{i+1/2}=1$ for all $i$.
As remarked in \cite{kpp}, many second- and higher-order schemes produce solutions that 
do not converge to the entropy solution.
To show this, we test a
fifth-order polynomial reconstruction which yields the interface values 
\begin{subequations}\label{non-WENO}
\begin{align}
  \hat u_{i}^+ &= \frac{1}{60}(-3u_{i-2} + 27u_{i-1} + 47u_i - 13u_{i+1} + 2u_{i+2}), \\
  \hat u_{i+1}^- &= \frac{1}{60}(2u_{i-2} - 13u_{i-1} + 47u_i + 27u_{i+1} - 3u_{i+2}).
\end{align}
\end{subequations}
The corresponding
 non-WENO semi-discretization is combined with the RK76 time integrator. The numerical solutions obtained with the resulting method on different mesh refinement levels are shown in Fig.~\ref{fig:kpp_non_WENO}.  Additionally, we perform a grid convergence study and summarize the results in Table \ref{table:kpp_non_WENO}. It can be seen that the non-WENO method based on \eqref{non-WENO} fails to converge to the entropy solution. None of the limiters presented in this work can fix this problem as long as the bounds are global and the baseline scheme is highly oscillatory. However, the combination of the fifth-order WENO space discretization with the RK76 time discretization does converge to the entropy solution even if no limiting is performed; see Fig.~\ref{fig:kpp_WENO_ExRK} and Table \ref{table:kpp_WENO_ExRK}. The application of GMC flux limiters removes the small overshoots and undershoots generated by the baseline schemes. The bound-preserving SSP54-GMC and RK76-GMC solutions are shown in Figs \ref{fig:kpp_SSP54_GMC} and \ref{fig:kpp_GMC_ExRK}, respectively. Tables~\ref{table:kpp_SSP54_GMC} and \ref{table:kpp_GMC_ExRK} summarize the results of our grid convergence studies for SSP54-GMC and RK76-GMC, respectively. 

\begin{figure}[!h]
  \centering
  \subfloat[RK76 (non-WENO) \label{fig:kpp_non_WENO}]{    
    \includegraphics[scale=0.5]{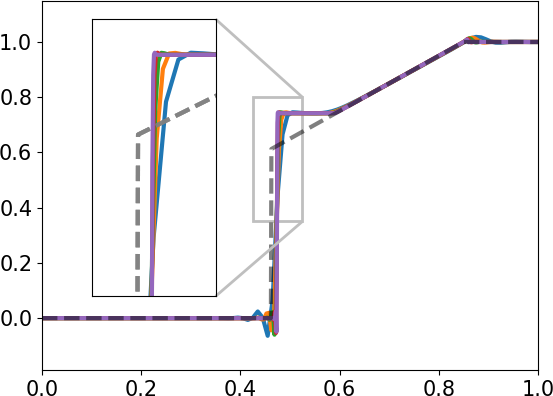}
  }
  \qquad
  \subfloat[RK76 (WENO) \label{fig:kpp_WENO_ExRK}]{    
    \includegraphics[scale=0.5]{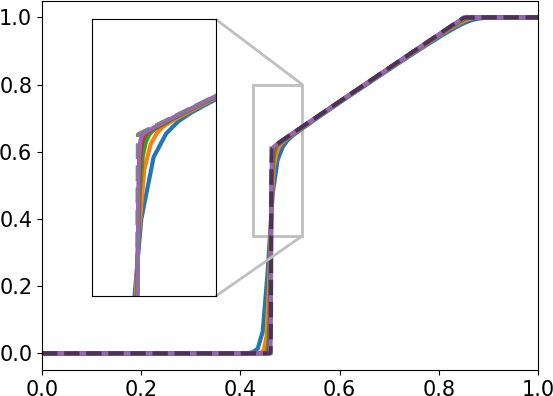}
  }

  \vspace{10pt}
  \subfloat[SSP54-GMC \label{fig:kpp_SSP54_GMC}]{    
    \includegraphics[scale=0.5]{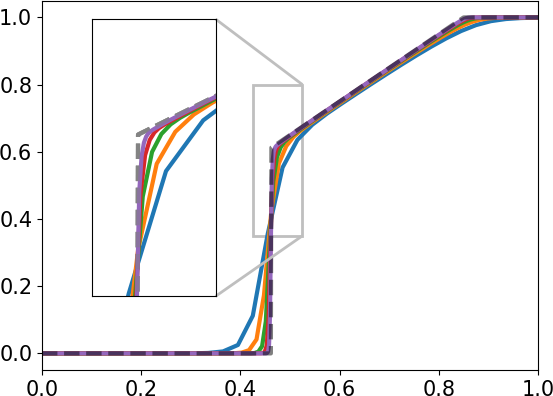}
  }
  \qquad
  \subfloat[RK76-GMC \label{fig:kpp_GMC_ExRK}]{    
    \includegraphics[scale=0.5]{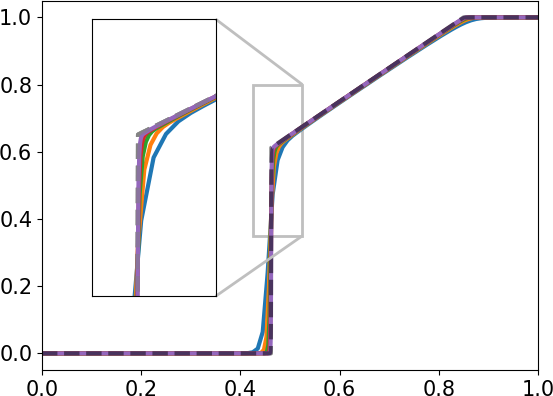}
  }
  \caption{Nonlinear problem \eqref{kpp} 
    with non-smooth initial data \eqref{kpp_init}.
    The exact solution at $t=1$ is shown as a dashed
    gray line. The remaining curves are numerical solutions
    obtained on different mesh refinement levels.
    \label{fig:kpp}}
\end{figure}

\begin{table}[!h]\scriptsize
  \begin{center}
    \subfloat[RK76 (non-WENO)\label{table:kpp_non_WENO}]{
      \begin{tabular}{||c|c|c|c||} \hline
        $N_h$ & $E_1$ & EOC & $\delta$  \\ \hline
        100  & 2.22e-02 &   --  & -1.33e-01 \\
        200  & 2.04e-02 & 0.118 & -1.34e-01 \\
        400  & 1.48e-02 & 0.463 & -1.34e-01 \\
        800  & 1.44e-02 & 0.041 & -1.34e-01 \\
        1600 & 1.40e-02 & 0.040 & -1.34e-01 \\ \hline
      \end{tabular}
    }
    \quad
    \subfloat[RK76 (WENO) \label{table:kpp_WENO_ExRK}]{
      \begin{tabular}{||c|c|c|c||} \hline
        $N_h$ & $E_1$ & EOC & $\delta$ \\ \hline
        100 & 2.84e-02 &   -- & -3.57e-09 \\
        200 & 1.28e-02 & 1.15 & -4.30e-09 \\
        400 & 7.29e-03 & 0.81 & -3.62e-09 \\
        800 & 3.80e-03 & 0.93 & -4.89e-09 \\
        1600& 1.98e-03 & 0.94 & -5.34e-09 \\ \hline
      \end{tabular}
    }

    \subfloat[SSP54-GMC \label{table:kpp_SSP54_GMC}]{
      \begin{tabular}{||c|c|c|c||} \hline
        $N_h$ & $E_1$ & EOC & $\delta$ \\ \hline
        100 & 2.84e-02 &   -- & -1.11e-15 \\
        200 & 1.28e-02 & 1.15 & -1.09e-14 \\
        400 & 7.29e-03 & 0.81 & -1.13e-14 \\
        800 & 3.80e-03 & 0.93 & -5.62e-14 \\
        1600 & 1.98e-03 & 0.94 & -5.80e-14 \\ \hline
      \end{tabular}
    }
    \quad
    \subfloat[RK76-GMC\label{table:kpp_GMC_ExRK}]{
      \begin{tabular}{||c|c|c|c|c||} \hline
        $N_h$ & $E_1$ & EOC & $\delta$ \\ \hline
        100  &  2.84e-02 &   -- & -1.11e-15 \\
        200  &  1.28e-02 & 1.15 & -1.18e-14 \\
        400  &  7.29e-03 & 0.81 & -1.07e-14 \\
        800  &  3.80e-03 & 0.93 & -5.31e-14 \\
        1600 &  1.98e-03 & 0.94 & -5.31e-14 \\ \hline
      \end{tabular}
    }
    \caption{Grid convergence study for the nonlinear problem \eqref{kpp} with non-smooth 
      data \eqref{kpp_init}. \label{table:kpp}}
  \end{center}
\end{table}


\section{Conclusions}\label{sec:conclusions}

The convex limiting approaches explored in this work are applicable
to a wide range of space discretizations combined with high-order
Runge-Kutta time-stepping methods. Although only explicit finite
volume schemes were considered in our numerical study, the same
flux correction tools can be used to constrain high-order finite
element discretizations and implicit RK methods. As shown in 
\cite{entropyDG,entropyCG,entropyHO},  convex
limiting in space makes it possible to enforce semi-discrete
entropy stability conditions in addition to maximum principles. 
Moreover, the new GMC limiter and its localized prototype
proposed in \cite{convex} belong to the family of monolithic
algebraic flux correction (AFC) schemes which lead to well-posed
nonlinear problems and are backed by theoretical analysis \cite{CL-diss}.

\section*{Acknowledgements}
The work of Dmitri Kuzmin and Johanna Gr\"ull was supported by the German Research Association (DFG) 
under grant KU 1530/23-1.
The work of Manuel Quezada de Luna and David I. Ketcheson was 
funded by King Abdullah University of Science and Technology (KAUST) in Thuwal, Saudi Arabia.

\newpage

\appendix
\section{High-order baseline RK methods}\label{appendix:RK_methods}

In this appendix, we provide details of the three high-order explicit 
baseline RK methods for solving
\beq
\td{u_i}{t} = F_i(\hat u):=-\frac{1}{|K_i|}\sum_{\jinN}|S_{ij}|H_{ij}(\hat u).
\eeq

\subsection{Fourth-order strong stability preserving (SSP54) RK method}\label{sec:ssp54}
The SSP-RK time integrator that we consider in this work is the 4th-order method
proposed in \cite{Kraaijevanger} and
\cite{Spiteri2002}. The Butcher form of its intermediate stages is as follows:
\begin{align*}
  y^{(1)} &= u^n + 0.391752226571890\Delta t F(u^n), \\
  y^{(2)} &= 0.444370493651235u^n + 0.555629506348765 y^{(1)}+0.368410593050371\Delta t F(y^{(1)}), \\
  y^{(3)} &= 0.620101851488403 u^n + 0.379898148511597 y^{(2)} + 0.251891774271694\Delta tF(y^{(2)}), \\
  y^{(4)} &= 0.178079954393132 u^n + 0.821920045606868 y^{(3)} + 0.544974750228521 \Delta t F(y^{(3)}), \\
  y^{(5)} &= 0.517231671970585 y^{(2)} 
  + 0.096059710526147 y^{(3)} + 0.063692468666290 \Delta t F(y^{(3)}) \\
  &+ 0.386708617503269 y^{(4)} + 0.226007483236906 \Delta t F(y^{(4)}). \nonumber
\end{align*} 
Note that each stage is a convex combination of Euler steps. Therefore, 
if $F(\cdot)$ is a BP spatial discretization and $u^n$ is BP, 
then each stage is BP under appropriate time step restrictions. The RK update is given by 
$u^{n+1} = y^{(5)}$. Hence, if the stages are BP, $u^{n+1}$ is BP and no extra limiting is needed. 

\subsection{Fifth-order extrapolated Euler (ExE-RK5) RK method}\label{appendix:ExE-RK5}
The Butcher tableau of the 5th-order extrapolated Euler RK method is given by 
{\small
\begin{align}\label{EeRK5}
\begin{tabular}{c|ccccccccccc}
  0 & 0 & & & & & & & & & &  \\
  1/2 & 1/2 & 0 & & & & & & & & &  \\
  1/3 & 1/3 & 0 & 0 & & & & & & & &  \\
  2/3 & 1/3 & 0 & 1/3 & 0 & & & & & & &  \\
  1/4 & 1/4 & 0 & 0 & 0 & 0 & & & & & &  \\
  1/2 & 1/4 & 0 & 0 & 0 & 1/4 & 0 & & & & &  \\
  3/4 & 1/4 & 0 & 0 & 0 & 1/4 & 1/4 & 0 & & & &  \\
  1/5 & 1/5 & 0 & 0 & 0 & 0 & 0 & 0 & 0 & & &  \\
  2/5 & 1/5 & 0 & 0 & 0 & 0 & 0 & 0 & 1/5 & 0 & &  \\
  3/5 & 1/5 & 0 & 0 & 0 & 0 & 0 & 0 & 1/5 & 1/5 & 0 &  \\
  4/5 & 1/5 & 0 & 0 & 0 & 0 & 0 & 0 & 1/5 & 1/5 & 1/5 & 0 \\ \hline
  & 0 & -4/3 & 27/4 & 27/4 & -32/3 & -32/3 & -32/3 & 125/24 & 125/24 & 125/24 & 125/24. \\ 
\end{tabular}
\end{align}
}

\medskip
The intermediate stages (written in Shu-Osher and Butcher form) are as follows:
\begin{align*}
y^{(1)} &= u^n  &\approx &u(t^n), \\
y^{(2)} &= y^{(1)}   + \frac{1}{2} \Delta t F(y^{(1)}) & \approx & u(t^n+\Delta t/2), \\
y^{(3)} &= y^{(1)}   + \frac{1}{3} \Delta t F(y^{(1)}) & \approx & u(t^n+\Delta t/3), \\ 
y^{(4)} &= y^{(3)}   + \frac{1}{3} \Delta t F(y^{(3)})
= y^{(1)}+\frac{\Delta t}{3} [F(y^{(1)}) + F(y^{(3)})] & \approx & u(t^n+2\Delta t/3), \\
y^{(5)} &= y^{(1)}   + \frac{1}{4} \Delta t F(y^{(1)}) & \approx & u(t^n+\Delta t/4), \\
y^{(6)} &= y^{(5)}   + \frac{1}{4} \Delta t F(y^{(5)})
= y^{(1)}+\frac{\Delta t}{4} [F(y^{(1)}) + F(y^{(5)})] & \approx & u(t^n+\Delta t/2), \\
y^{(7)} &= y^{(6)}   + \frac{1}{4} \Delta t F(y^{(6)})
= y^{(1)}+\frac{\Delta t}{4} [F(y^{(1)}) + F(y^{(5)}) + F(y^{(6)})] & \approx & u(t^n+3\Delta t/4), \\
y^{(8)} &= y^{(1)}   + \frac{1}{5} \Delta t F(y^{(1)}) & \approx & u(t^n+\Delta t/5), \\
y^{(9)} &= y^{(8)}   + \frac{1}{5} \Delta t F(y^{(8)})
= y^{(1)}+\frac{\Delta t}{5} [F(y^{(1)}) + F(y^{(8)})] & \approx & u(t^n+2\Delta t/5), \\
y^{(10)} &= y^{(9)}  + \frac{1}{5} \Delta t F(y^{(9)})
= y^{(1)}+\frac{\Delta t}{5} [F(y^{(1)}) + F(y^{(8)}) + F(y^{(9)})] &\approx & u(t^n+3\Delta t/5), \\
y^{(11)} &= y^{(10)} + \frac{1}{5} \Delta t F(y^{(10)})
= y^{(1)}+\frac{\Delta t}{5} [F(y^{(1)}) + F(y^{(8)}) + F(y^{(9)}) + F(y^{(10)})]& \approx & u(t^n+4\Delta t/5). 
\end{align*}
Note that if $F(\cdot)$ is a BP spatial discretization and $ u^n$ is BP, 
then each stage of this ExE-RK method
is BP under appropriate time step restrictions. The approximations $y^{(1)}$, $y^{(2)}$, $y^{(3)}$, $y^{(5)}$, and $y^{(8)}$ are BP because they correspond to forward Euler updates of $u^n$. The remaining stages are BP since $y^{(m)}$ is a forward Euler update of a BP approximation
$y^{(r)}$ for some $r\in\{1,\ldots,m-1\}$.

The Aitken-Neville interpolation yields the temporally
5th-order approximation
\begin{align*}
u^{\rm RK} = 
&\frac{1}{24}\left[y^{(1)}+\Delta t F(y^{(1)})\right]
-\frac{8}{3}\left[y^{(2)}+\frac{1}{2}\Delta t F(y^{(2)})\right]
+\frac{81}{4}\left[y^{(4)}+\frac{1}{3}\Delta t F(y^{(4)})\right]\\
&-\frac{128}{3}\left[y^{(7)}+\frac{1}{4}\Delta t F(y^{(7)})\right]
+\frac{625}{24}\left[y^{(11)}+\frac{1}{5}\Delta t F(y^{(11)})\right].
\end{align*}
Note that this Euler extrapolation method combines $S=5$ first-order approximations of $u^{n+1}$. 
Since this combination is not convex, $u^{\rm RK}$ is not necessarily BP even if
$y^{(1)},\ldots,y^{(11)}$ are BP. To
enforce the BP property of the final solution, we perform flux limiting using
the Butcher form representation
\begin{align}\label{update}
u^{\rm RK} = u^n + \Delta t 
\Big[
  &-\frac{4}{3}F(y^{(2)}) + \frac{27}{4}F(y^{(3)}) + \frac{27}{4}F(y^{(4)}) \nonumber
  - \frac{32}{3}F(y^{(5)}) - \frac{32}{3}F(y^{(6)}) -\frac{32}{3}F(y^{(7)}) \\
  &+ \frac{125}{24}F(y^{(8)}) + \frac{125}{24}F(y^{(9)}) + \frac{125}{24}F(y^{(10)} + \frac{125}{24}F(y^{(11)})
\Big].
\end{align}

\subsection{Sixth-order (RK76) RK method}\label{sec:RK76}
This 6th-order RK method, proposed in \cite{rkmethods}, consists of seven stages and has the Butcher tableau
\begin{align}\label{RK76}
\begin{tabular}{c|ccccccc}
  0 & 0 & & & & & & \\
  1/3 & 1/3 & 0 & & & & & \\
  2/3 & 0 & 2/3 & 0 & & & & \\
  1/3 & 1/12 & 1/3 & -1/12 & 0 & & & \\
  1/2 & -1/16 & 9/8 & -3/16 & -3/8 & 0 & & \\
  1/2 & 0 & 9/8 & -3/8 & -3/4 & 1/2 & 0 & \\
  1   & 9/44 & -9/11 & 63/44 & 18/11 & 0 & -16/11 & 0\\ \hline
  & 11/120 & 0 & 27/40 & 27/40 & -4/15 & -4/15 & 11/120.
\end{tabular}
\end{align}
The intermediate stages of this method are not Euler steps. If we require them to be BP, the numerical fluxes $H_{ij}$ should be constrained using the limiters from Section \ref{sec:stagewise} in each stage. The BP property of the final solution $u^{\rm RK}$ can be enforced similarly using the limiter from Section \ref{sec:limiting_last}.

\section{Flux-limited WENO discretization in 1D}\label{appendix:one-dim_GMC_ExRK}
In this appendix, we provide details of flux-corrected RK methods for one-dimensional
hyperbolic conservation laws of the form $\pd{u}{t}+\pd{f(u)}{x}=0$. Although the
underlying low-order and high-order approximations are of little interest
{\it per se}, we present their one-dimensional formulations as well.
We assume the mesh is uniform and, therefore, the mesh size
$|K_i|=\Delta x$ is constant.

\subsection{The low-order method}\label{appendix:low_order}

In one space dimension, the common interface $S_{ij}$ of control
volumes with indices $i$ and
$j=i+1$ is the point $x_{i+1/2}=\frac12(x_i+x_{i+1})$. The 1D version of
the FE-LLF approximation \eqref{bp-update} is given by
\begin{align}\label{low_order_1D}
  u_i^{\rm FE}=u_i^n - \frac{\Delta t}{\Delta x}[H^{\rm FE}_{i+1/2}-H^{\rm FE}_{i-1/2}],
\end{align}
where 
\begin{align*}
  H^{\rm FE}_{i+1/2}=\frac{f(u_i^n)+f(u_{i+1}^n)}{2} -\lambda_{i+1/2}
  \frac{u_{i+1}^n - u_i^n}{2}
\end{align*}
is the first-order numerical flux and $\lambda_{i+1/2}$ is an upper bound
for the wave speed of the Riemann problem associated with the states
$u_i^n$ and $u_{i+1}^n$. The bar state form
of \eqref{low_order_1D}  is given by
\begin{align*}
  u_i^{\rm FE}=u_i^n + \frac{\Delta t}{\Delta x}[\lambda_{i+1/2}(\bar{u}^L_{i+1/2}-u_i^n)+\lambda_{i-1/2}(\bar{u}^L_{i-1/2}-u_i^n)]
=u_i^n + \frac{\Delta t}{\Delta x}d_i(\bar u_i^L-u_i^n),
\end{align*}
where
\begin{gather*}
  \bar u_{i+1/2}^L = \frac{u_{i+1}^n+u_i^n}{2} 
  - \frac{f(u_{i+1}^n)-f(u_i^n)}{2\lambda_{i+1/2}}, \qquad
  \bar u_{i-1/2}^L = \frac{u_i^n+u_{i-1}^n}{2} 
  + \frac{f(u_{i-1}^n)-f(u_i^n)}{2\lambda_{i-1/2}},\\
  \bar u_i^L=\frac{1}{d_i}[\lambda_{i+1/2}\bar u_{i+1/2}^L
  +\lambda_{i-1/2}\bar u_{i-1/2}^L],\qquad
  d_i = \lambda_{i+1/2}+\lambda_{i-1/2}.
\end{gather*}

This representation proves that the 
 method is bound-preserving (BP) with respect to the local bounds 
$u_i^{\max}=\max\{u_{i-1}^n,u_i^n,u_{i+1}^n\}$ and $u_i^{\min}=\min\{u_{i-1}^n,u_i^n,u_{i+1}^n\}$,
provided  
\beq
\Delta t\leq \frac{\Delta x}{\lambda_{i+1/2}
  +\lambda_{i-1/2}}.
\eeq
However, the method is only first-order accurate in space and time.

\subsection{The baseline high-order method}\label{appendix:rk_method}

The unlimited form of an explicit Runge-Kutta method (with a WENO discretization) is given by 
\begin{align}\label{high_order_1D}
  u_i^{\rm RK}=u_i^n-\frac{\Delta t}{\Delta x}[H_{i+1/2}^{\text{RK}}-H_{i-1/2}^{\text{RK}}],\qquad   H_{i+1/2}^{\text{RK}} = \sum_{m=1}^M b_m H_{i+1/2}^{(m)},
\end{align}
where 
$M$ is the number of stages of the RK method and $b_m$ are the Butcher weights of the final RK update. The  Lax-Friedrichs flux $H_{i+1/2}^{(m)}=H\big(\hat y_i^{(m)}(x_{i+1/2}),\hat y_{i+1}^{(m)}(x_{i+1/2}),1\big)$ is calculated using the interface values of the high-order WENO reconstructions $\hat y^{(m)}_i(x)$ and $\hat y^{(m)}_{i+1}(x)$ in cells $K_{i}$ and $K_{i+1}$, respectively. 
These WENO reconstructions are obtained from the $m$-th stage 
cell averages
\beq
y_i^{(m)} = u_i^n-\frac{\Delta t}{\Delta x}
\sum_{s=1}^{m-1}
a_{ms}\big[H_{i+1/2}^{(s)}-H_{i-1/2}^{(s)}\big],
\eeq
where $a_{ms},\ s=1,\ldots,m-1$ are the coefficients of the $m$-th row 
in the Butcher tableau \eqref{butcher_tableau}. 
Method \eqref{high_order_1D} is high-order in space and time. In particular, 
we combine a fifth-order WENO spatial discretization with the three explicit RK methods in 
\ref{appendix:RK_methods}.
The use of WENO numerical fluxes produces a solution which is typically (almost) non-oscillatory. 
However, this solution is not BP in general. 

\subsection{Spatial GMC limiting in 1D}\label{sec:1Dspatial}
Let us now apply the GMC space limiters from Section \ref{sec:gmcl}
to the high-order semi-discretization 
\beq\label{high_order_semi_1D}
\Delta x \td{u_i}{t}
= -[H_{i+1/2}^{\rm WENO}-H_{i-1/2}^{\rm WENO}],
\eeq
where $H_{i+1/2}^{\rm WENO}=H(\hat u_i(x_{i+1/2}),\hat u_{i+1}(x_{i+1/2}),1)$ is the high-order WENO flux. Using the corresponding low-order
flux  $H_{i+1/2}^{\rm LLF}=H(u_i,u_{i+1},1)$, the baseline scheme
\eqref{high_order_semi_1D} can be written as 
\beq
\Delta x\td{u_i}{t} 
= -\big[\big(H_{i+1/2}^{\rm LLF}-F_{i+1/2}\big)
-\big(H_{i-1/2}^{\rm LLF}-F_{i-1/2}\big)\big].
\eeq
The one-dimensional GMC limiter
replaces the antidiffusive
flux $F_{i+1/2}=H_{i+1/2}^{\rm LLF}-H_{i+1/2}^{\rm WENO}$
by its limited counterpart
$F_{i+1/2}^*=\alpha_{i+1/2}F_{i+1/2}$. The
correction factor $\alpha_{i+1/2}$ is calculated 
as follows: 
\begin{subequations}\label{GMC_alpha_1D}
\begin{alignat}{2}
  P^+_i &= \max\{0,F_{i+1/2}\} + \max\{0,-F_{i-1/2}\}, &
  P^-_i &= \min\{0,F_{i+1/2}\} + \min\{0,-F_{i-1/2}\},\label{GMCP} \\
  Q_i^{+} &= d_i(u^{\max}_i-\bar u_i^L)+\gamma(u^{\max}_i-u_i^n)
,&
Q_i^{-} &= d_i(u^{\min}_i-\bar u_i^L)
+\gamma d_i(u^{\min}_i-u_i^n),\\
R^\pm_i &= \begin{cases} 
  1 &\mbox{if } \phantom{|}P^\pm_i\phantom{|}=0, \\
  \min\left\{1, \frac{Q_i^\pm}{P_i^\pm}\right\} & \mbox{if } |P^\pm_i|>0,
\end{cases} &
 \alpha_{i+1/2}&=
  \begin{cases}
    \min\{R^+_i,R^-_{i+1}\} & \mbox{ if }\ F_{i+1/2}> 0, \\
    \min\{R^-_i,R^+_{i+1}\} & \mbox{ if }\ F_{i+1/2}\le 0.
  \end{cases} \label{GMCR}
\end{alignat}
\end{subequations}
An SSP-RK time discretization of the flux-corrected WENO scheme
\beq
\Delta x\td{u_i}{t} 
= -\big[\big(H_{i+1/2}^{\rm LLF}-F_{i+1/2}^*\big)
-\big(H_{i-1/2}^{\rm LLF}-F_{i-1/2}^*\big)\big]
\eeq
is BP under the time step restriction
\beq\label{cfl1D}
\Delta t\leq \frac{\mu(1+\gamma)\Delta x}{\lambda_{i+1/2}
  +\lambda_{i-1/2}},
\eeq
where $\mu\in(0,1]$ is the SSP coefficient of the RK method
and $\gamma\ge 0$ is the GMC
  relaxation parameter.
  
\subsection{Space-time GMC limiting in 1D}\label{appendix:space_and_time_gmcl}
The final stage  \eqref{high_order_1D} of a baseline
high-order WENO-RK scheme  can be written as
\beq\nonumber
u_i^{\rm RK}=u_i^n-\frac{\Delta t}{\Delta x}\big[H_{i+1/2}^{\text{RK}}-H_{i-1/2}^{\text{RK}}\big]
=
u_i^n - \frac{\Delta t}{\Delta x}\big[(H^{\rm FE}_{i+1/2}-F^{\rm RK}_{i+1/2})
  -(H^{\rm FE}_{i-1/2}-F^{\rm RK}_{i-1/2}\big)\big],
\eeq
where $F^{\rm RK}_{i+1/2}=H_{i+1/2}^{\FE}-H_{i+1/2}^{\RK}$ is the raw antidiffusive flux. Using this representation, space-time GMC limiting can be performed as in \ref{sec:1Dspatial} using $H^{\rm FE}_{i+1/2}$ in place of $H^{\rm LLF}_{i+1/2}$ and $F^{\rm RK}_{i+1/2}$ in place of $F_{i+1/2}$. The implementation of GMC for intermediate RK stages is similar.

\begin{rmk}
  The FCT version of the space-time limiter uses $u_i^{\rm FE}$ defined by
  \eqref{low_order_1D} to construct
  \beq\nonumber
  Q_i^{+} = \frac{\Delta x}{\Delta t}(u^{\max}_i-u_i^{\rm FE}), \qquad
  Q_i^{-} = \frac{\Delta x}{\Delta t}(u^{\min}_i-u_i^{\rm FE})
  \eeq
  for calculation of the correction factors $R_i^\pm$ and
  $\alpha_{ij}$ in \eqref{GMCR}.
  
\end{rmk}

\end{document}